\documentclass[11pt,reqno]{amsart}

\usepackage{amsmath,amssymb}
\usepackage{graphicx}
\usepackage[pagebackref, colorlinks = true, linkcolor = blue, urlcolor  = blue, citecolor = red]{hyperref}
\usepackage{xcolor}

\usepackage[margin=1in]{geometry}

\usepackage{caption} 
\usepackage{subcaption} 

\newcommand{\eq}[1]{\begin{align}#1\end{align}}
\newcommand{\eqb}[1]{\begin{align}\begin{aligned}#1\end{aligned}\end{align}}

\newtheorem{theorem}{Theorem}[section]
\newtheorem*{theorem*}{Theorem}
\newtheorem{definition}[theorem]{Definition}
\newtheorem{proposition}[theorem]{Proposition}
\newtheorem{corollary}[theorem]{Corollary}
\newtheorem{lemma}[theorem]{Lemma}
\newtheorem{remark}[theorem]{Remark}

\DeclareMathOperator{\trace}{Tr} 
\DeclareMathOperator{\cat}{Cat}

\begin{document}

\date{\today}

\title{Enumerating meandric systems with large number of loops}

\author{Motohisa Fukuda}
\address{MF: Yamagata University, 1-4-12 Kojirakawa, Yamagata, 990-8560 Japan}
\email{fukuda@sci.kj.yamagata-u.ac.jp}

\author{Ion Nechita}
\address{IN: Zentrum Mathematik, M5, Technische Universit\"at M\"unchen, Boltzmannstrasse 3, 85748 Garching, Germany
and CNRS, Laboratoire de Physique Th\'{e}orique, IRSAMC, Universit\'{e} de Toulouse, UPS, F-31062 Toulouse, France}
\email{nechita@irsamc.ups-tlse.fr}

\subjclass[2000]{}

\begin{abstract}
We investigate meandric systems with large number of loops using tools inspired by free probability.
For any fixed integer $r$, we express the generating function of meandric systems on $2n$ points with $n-r$ loops in terms of a finite (the size depends on $r$) subclass of irreducible meandric systems,
via the moment-cumulant formula from free probability theory. 
We show that the generating function, after an appropriate change of variable, is a rational function, and we bound its degree.  Exact expressions for the generating functions are obtained for $r \leq 6$, as well as the asymptotic behavior of the meandric numbers for general $r$. 
\end{abstract}

\maketitle

\tableofcontents

\section{Introduction}
In this paper we consider the problem of enumerating meandric systems, which are a natural generalizations of meanders. This problem falls into the category of enumerating some non-crossing diagrams, and it has received a lot of interest from the mathematics and physics communities. An excellent reference providing an extensive overview of the problem and its various connections to various branches of mathematics is \cite{difrancesco1997meander}.

\emph{Meandric systems} of order $n$  are defined as non-crossing closed loops which intersect a straight infinite line at $2n$ points.  
To configure all possible shapes, one can draw a horizontal line with $2n$ points and 
choose two non-crossing pair partitions on those $2n$ points for the upper and lower sides of the line so that 
connecting them gives a set of closed loops. 
In Figure \ref{fig:meander_eg}, two non-crossing parings $\{(1,2),(3,6),(4,5)\}$ and $\{(1,4),(2,3),(5,6)\}$
result in a meandric system with one loop, which is simply called a \emph{meander}. 
For fixed $n$, we can define $M_n^{(k)}$ to be the number of  meandric systems with $2n$ fixed points and $k$ loops. 
Computing $M_n^{(1)}$, the number of meanders, is a notoriously difficult problem, while $M_n^{(n)}$ is just the \emph{Catalan number} $\cat_n$, because such a meandric system is obtained when the same non-crossing pair partitions (or arches) are chosen for the upper and lower diagrams. The example in Figure \ref{fig:meander_eg} contributes to $M_3^{(1)}$. The problem of enumerating meanders and meandric systems is an important one, and has received a lot of attention in the last three decades \cite[Section 6]{sloane1999my}.

\begin{figure}[htbp]
  \centering
    \includegraphics[width=0.4\textwidth]{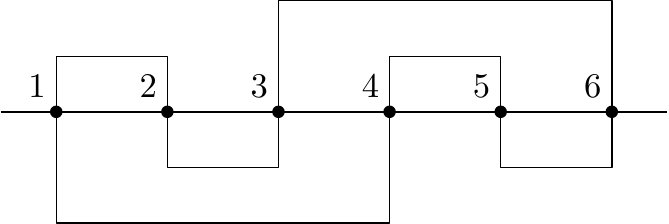}
\caption{A meander by $\{(1,2),(3,6),(4,5)\}$ and $\{(1,4),(2,3),(5,6)\}$.}
 \label{fig:meander_eg}
\end{figure}

One of main goals of research on meandric systems is to find explicit formulas for the numbers $M_n^{(k)}$ for any $k,n \in \mathbb N$. With the help of computers, these numbers (sequence \href{http://oeis.org/A008828}{A008828} in \cite{oeis2011online}) have been computed up to $n$ of order $30$, see e.g.~\cite{bobier2010fast}. 
In this paper, we focus on the formulas for $M_n^{(n-r)}$ for fixed $r\in \mathbb N$ and any $n \in \mathbb N$.
Such formulas were obtained for $0\leq r \leq 5$ in \cite{difrancesco1997meander}, where the authors claimed to have proved them for $0\leq r \leq 3$. In this work, we obtain the general formula for the generating functions of these numbers, and the exact values for $r \leq 6$. Moreover, we introduce a recipe for obtaining the generating function for any given $r$, which we implement on a computer algebra system, see \cite{num}. Before moving on, let us be clear that we do not touch on what is probably considered the most important problem in the field, the enumeration of meanders, that is the numbers $M_n^{(1)}$ (sequence \href{http://oeis.org/A005315}{A005315} in \cite{oeis2011online}). 

Our approach consists in translating the problems of meandric systems  
into problems about non-crossing partitions and permutations, 
by using the bijection between non-crossing parings of $2n$ points and the so-called geodesic permutations of $n$ elements. 
These relations between meanders and permutations were investigated in \cite{sav,hal06}, and
precise and detailed discussions will be made in Section \ref{sec:M-NC}. A key concept in our analysis is the notion of \emph{irreducible meandric systems}, which were introduced in \cite{lando1993plane} (sequence \href{http://oeis.org/A006664}{A006664} in \cite{oeis2011online}); the idea of counting combinatorial objects in terms of ``irreducibles'' dates back to Beissinger \cite{beissinger1985enumeration}.  
However, we introduce the parameter $r$ in the study of irreducible meandric systems to analyze the number of loops of a meandric system. 
Recently, Nica \cite{nica2016free} also analyzed irreducible meandric systems from the point of view of free probability, but with different goals and methods than ours. 
 The focus in \cite{nica2016free} is on the set of irreducible meandric systems, independently of the parameter $r$, while $r$ plays a key role in our investigations. 
We refer the reader to the comments at the end of Section \ref{sec:M-I} for a discussion of the similarities and differences between these two papers. 

Inspired by the language of free probability, we show that the generating function of the sequence $(M_n^{(n-r)})_n$ for fixed $r$ can be obtained from that for irreducible meandric systems, through natural transformations between moments and free cumulants (Theorem \ref{thm:M-I}). 
Since  irreducible meandric systems with fixed parameter $r$ live on at most $n = 2r$ points, we can derive the generating function as described above. 
The following statement is our main result in this paper 
(see Theorem \ref{thm:meanders-generating-function} for the precise statement), providing an alternative answer to the conjecture in \cite[Equation (2.4)]{difrancesco1997meander} (see also Remark \ref{rem:rel-conj}):

\begin{theorem}\label{thm:main-informal}
Let $F_r$ be the generating function of the number of meanders on $2n$ points with $n-r$ loops 
$$F_r(t) = \sum_{n=r+1}^\infty M_n^{(n-r)}t^n.$$
Then, with the change of variables $t=w/
(1+w)^2$, the functions $F_r$ read

\begin{equation}
F_r(t) = \sum_{n=r+1}^\infty  M_n^{(n-r)} \frac{w^n}{(1+w)^{2n}} = \frac{w^{r+1}(1+w)}{(1-w)^{2r-1}} \tilde P_r(w),
\end{equation}
where $\tilde P_r(w)$ are polynomials of degree at most $3(r-1)$ (see Section \ref{sec:numerics} for the values of these polynomials up to $r=6$).
\end{theorem} 

The paper is organized as follows. In Section \ref{sec:nc-partitions-permutations} we recall some basic properties of non-crossing partitions and permutations, which are used in Section \ref{sec:M-NC} to make the connection to meandric systems. 
In Section \ref{sec:irred} we introduce irreducible meandric systems, together with three parameters which are going to be used later for enumerating meandric systems. 
Section \ref{sec:M-I} contains the main result of the paper;
we obtained the general form of the generating function of meandric systems with large number of loops, as well as their asymptotic behavior,
after establishing their relations to irreducible meandric systems via the moment-cumulant formula. The first few exact values of the polynomials appearing in the generating functions are presented in Section \ref{sec:numerics}.
 
\bigskip

\noindent\textbf{Acknowledgements.} The authors would like to thank the anonymous reviewers for very useful remarks that helped us improve the quality of the presentation. M.F. was financially supported by the CHIST-ERA/BMBF project CQC and JSPS KAKENHI Grant Number JP16K00005. I.N.'s research has been supported by a von Humboldt fellowship and by the ANR projects {OSQPI} {2011 BS01 008 01},  {RMTQIT}  {ANR-12-IS01-0001-01}, and {StoQ} {ANR-14-CE25-0003-01}. Both authors acknowledge the hospitality of the Mathematical Physics group of the Technische Universit\"at M\"unchen, where this research was initiated. 

\section{Combinatorial aspects of non-crossing partitions and permutations}\label{sec:nc-partitions-permutations}

We gather in this section some well-known definitions and facts about non-crossing partitions and non-crossing permutations. Many of these facts are folklore, but one can follow \cite{biane1997some} or the excellent monograph \cite{nica2006lectures}. 

For a permutation $\alpha$ in the symmetric group $S_n$, we denote by $\|\alpha\|$ its length, 
that is the minimal number $m$ of transpositions $\tau_1, \ldots ,\tau_m$ which multiply to $\alpha$:
$$\|\alpha\| := \min\{m \geq 0 \, : \, \exists \, \tau_1, \ldots, \tau_m \in \mathcal S_n \text{ transpositions s.t.~}
 \alpha = \tau_1 \cdots \tau_m\}.$$
Sometimes the notation $| \cdot |$ is used for the length of permutations, but 
in our paper it stands for the cardinality of a set.
The length function $\| \cdot \|$ induces a distance on $\mathcal S_n$ by $d(\alpha, \beta) = \|\alpha^{-1}\beta\|$.
Importantly, we have
\eq{
\|\alpha\| + \#(\alpha) = n
}
where $\#(\alpha)$ is the number of cycles in $\alpha$. Note also that both $\#(\cdot)$ and $\| \cdot \|$ are class functions (with respect to conjugation) and thus, e.g.~$\|\alpha\| = \|\alpha^{-1}\|$ and $\|\alpha \beta \| = \|\beta \alpha\|$. 

We recall next the concept of non-crossing partitions. A partition $A_1 \sqcup A_2 \sqcup \cdots \sqcup A_m = \{1,2,\ldots,n\}$ is called \emph{non-crossing} if
there do not exist $a,b \in A_i$ and $c,d \in A_j$ ($i \not =j$) such that $a<c<b<d$.
An example of non-crossing partition $\{1,4,5\} \sqcup \{2,3\}$ is found in Figure \ref{fig:NC1},
and crossing $\{ 1,3\}\sqcup\{2,4,5\}$ in Figure \ref{fig:NC2}. The set of non-crossing partitions of $\{1,2,\ldots,n\}$ is denoted by $NC(n)$ or $NC(1,2,\ldots,n)$. 
We are sometime interested in $NC(n)$ restricted only to partitions with blocks of size $2$, 
and we denote by $NC_2(n)$ this subset of $NC(n)$: $NC_2(n)$ is the set of non-crossing parings and $n$ must be an even number. 

Non-crossing partitions of $\{ 1,2,\ldots,n\}$ are naturally identified to a subset of permutations in $\mathcal S_n$, called \emph{geodesic (or non-crossing) permutations}, see \cite{biane1997some} or \cite[Lecture 23]{nica2006lectures}. The bijection corresponding to this identification associates to each block of a non-crossing partition a cycle in a permutation where the elements are ordered increasingly;
the example in Figure \ref{fig:NC1} is identified to the permutation $(1,4,5)(2,3) \in \mathcal S_5$. As it was 
shown by Biane in \cite{biane1997some}, geodesic permutations are characterized by the fact that they saturate the triangle inequality; $\alpha \in \mathcal S_n$ is geodesic iff
$$\|\alpha\| + \|\alpha^{-1} \pi \| = \|\pi\| = n-1,$$
where $\pi = (1,2, \ldots n)$ is the full-cycle permutation; we say that $\alpha$ lies on the geodesic between $\mathrm{id}=(1)(2)\cdots(n)$ and $\pi=(1,2,\ldots,n)$ in $\mathcal S_n$. We shall use the identification between non-crossing partitions and geodesic permutations. We write $\alpha^{-1}$ for a non-crossing partition $\alpha \in NC(n)$;  this should be understood as the permutation $\alpha^{-1} \in \mathcal S_n$ which lies on the geodesic between $\mathrm{id}$ and $\pi^{-1} = (n,n-1,\ldots,1)$.

\begin{figure}[htbp]
\centering
\begin{minipage}{.5\textwidth}
  \centering
  \includegraphics[width=0.8\linewidth]{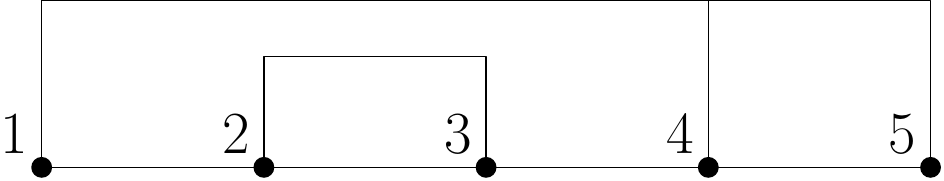}
  \captionof{figure}{A non-crossing partition.}
  \label{fig:NC1}
\end{minipage}\begin{minipage}{.5\textwidth}
  \centering
  \includegraphics[width=0.8\linewidth]{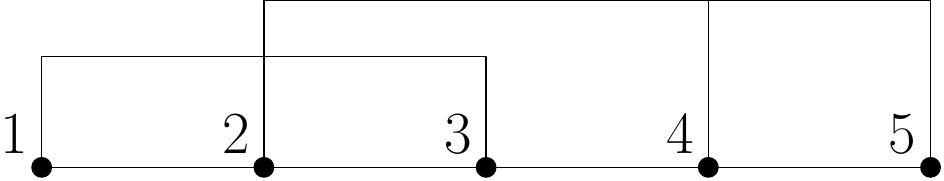}
  \captionof{figure}{A crossing partition.}
  \label{fig:NC2}
\end{minipage}
\end{figure}

Next, we recall the notion of \emph{Kreweras complement} for non-crossing partitions. 
The Kreweras complement of $\alpha \in NC(n)$ is another non-crossing partition, 
denoted $\alpha^\mathrm{Kr} \in NC(n)$, defined in the following way \cite[Definition 9.21]{nica2006lectures}. 
First, expand the domain of partitions to $\{1, \bar 1,  2,\bar2 \ldots,  n,\bar  n\}$
and
let then  $\alpha^\mathrm{Kr} \in NC(\bar 1,\bar 2, \ldots, \bar n) \cong NC(n)$ be 
the largest non-crossing partition, with respect to the partial order defined in the next paragraph, such that $\alpha \sqcup \alpha^\mathrm{Kr}$ is 
still a non-crossing partition on $\{1,\bar 1 ,2, \bar2 \ldots ,  n, \bar n\}$. 
In the 
language of geodesic permutations, given a geodesic permutation $\mathrm{id}-\alpha-\pi$, 
we define the geodesic permutation $\alpha^\mathrm{Kr} \in \mathcal S_n$ as $\alpha^\mathrm{Kr} = \alpha^{-1} \pi$ (see \cite[Remark 23.24]{nica2006lectures}). An example of Kreweras complement is found in Figure \ref{fig:Krew2}; we set $\alpha = (2,6)(3,4)$ and $\alpha^{\mathrm{Kr}}=(1,6)(2,4,5)$. Other trivial examples are $\mathrm{id}^{\mathrm{Kr}}=\pi$ and $\pi^{\mathrm{Kr}}=\mathrm{id}$.

The set $NC(n)$ is endowed with the partial order of \emph{reversed refinement}: $\alpha \leq \beta$ if every block of $\alpha$ is contained in a block of $\beta$. 
On the level of geodesic permutations, the partial order $\alpha \leq \beta$ is equivalent to $\alpha$  being on the geodesic between $\mathrm{id}$ and $\beta$: 
$$\|\alpha\| + \|\alpha^{-1} \beta\| = \|\beta\|.$$
Since $NC(n)$ is a lattice, for any $\alpha,\beta \in NC(n)$ we denote by $\alpha \wedge \beta$ the \emph{meet} of $\alpha$ and $\beta$, that is the largest element $\gamma \in NC(n)$ such that $\gamma \leq \alpha,\beta$. Similarly, we write $\alpha \vee \beta$ for the uniquely defined \emph{join} of $\alpha$ and $\beta$. Taking Kreweras complements, we have $(\alpha \wedge \beta)^\mathrm{Kr}=\alpha^\mathrm{Kr} \vee \beta^\mathrm{Kr}$ and 
$(\alpha \vee \beta)^\mathrm{Kr}=\alpha^\mathrm{Kr} \wedge \beta^\mathrm{Kr}$. The smallest element in $NC(n)$ is denoted by $0_n$ and corresponds to the partition made up of singletons, or to the identity permutation. The largest element of $NC(n)$ is denoted by $1_n$ and corresponds to the $1$-block partition, or to the permutation $\pi \in \mathcal S_n$ defined previously. 

\begin{figure}[htbp]
  \centering
    \includegraphics[width=0.6\textwidth]{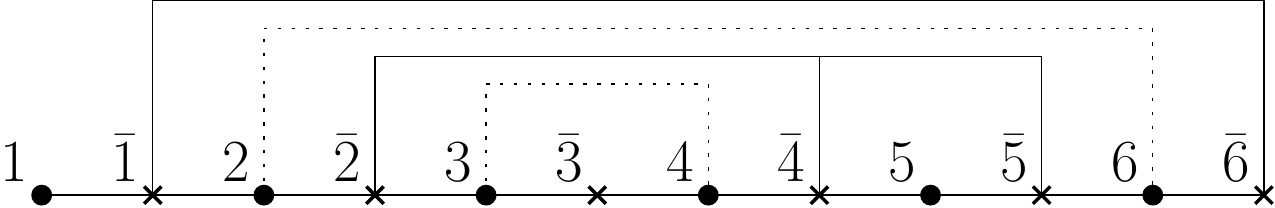}
\caption{The Kreweras complement of $(2,6)(3,4)$.}
 \label{fig:Krew2}
\end{figure}

Finally, let us discuss the well-known bijection between $NC(n)$ and $NC_2(2n)$, called \emph{fattening}. For a given non-crossing partition $\alpha \in NC(n)$, we consider two points $i_{-}$ and $i_{+}$ for both sides of each $i \in \{1,\ldots,n\}$,
left and right respectively, doubling the size on the index set. Associate now to $\alpha$ the following pairing: connect $i_{+}$ and $j_{-}$ if $\alpha(i) =j$, where $\alpha$ is seen now as a permutation. In can be shown that the pair partition obtained is also non-crossing, see \cite[Lecture 9]{nica2006lectures}.
We state a lemma on the fattening operation which is used later in the paper. 
\begin{lemma}\label{lemma:fat}
Take $\alpha \in NC(n)$ and denote its fattening by $\bar \alpha \in NC_2(2n)$. 
\begin{enumerate}
\item Suppose $(i_-,j_+) $ with $i\leq j$ is a pair in $\bar \alpha$. Then:
\begin{enumerate}
\item if $i=j$ implies that $i$ is a fixed point of the geodesic permutation $\alpha$
\item otherwise, $i<j$ and $\alpha$ has a cycle of the form $(i, \ldots, j)$ where the numbers in the bracket are in the increasing order.   
\end{enumerate}
\item Suppose $(i_+,j_-) $ with $i < j$ is a pair in $\bar \alpha$. Then, $\alpha$ has a cycle of the form $(\ldots i, j \ldots)$ where the numbers in the bracket are in the increasing order.   
\end{enumerate} 
\end{lemma}
\begin{proof}
To show (1)-(a), notice that the paring $(i_-,i_+)$ implies that $i \mapsto i$ by the definition of fattening. 
Similarly, the paring $(i_-,j_+)$ implies that $j \mapsto i$, but since the permutation $\alpha$ is on the geodesic between $\mathrm{id}$ and $\pi = (1,2,\ldots,n)$ the claim is proved. 
The claim (2) also follows from the definition. 
\end{proof}

We end the combinatorial treatment of non-crossing partitions by giving the reader a taste of the connection between pairs of non-crossing partitions and meandric systems. These fact will be treated rigourously and in detail in Section \ref{sec:M-NC}. 
In Figure \ref{fig:permutation_eg}, one can find that the meander of Figure \ref{fig:meander_eg} is represented by drawing
two geodesic permutations $\{(1),(2,3)\}$ and $\{(1,2),(3)\}$, one above and one below a horizontal line (the dotted lines show the fattening of the corresponding permutations).
In Figure \ref{fig:permutation2_eg}, the fattening operation is drawn with directions,
which show how the original permutations act. 
Interestingly, the arrows from the upper and lower sides of the horizontal line are conflicting.
However, if one inverts the directions of the lower side, one can have a loop with consistent directions. 
In fact, such a loop results from the following calculation
$$(1)(2,3) \circ \left\{(1,2)(3) \right\}^{-1} = (1,3,2)$$
which is equivalent to the fact that only one meander is generated from such pairs of permutations. 
Further details are found in Section \ref{sec:M-NC} and
one can understand more of this concept going through the example and the proof of Proposition \ref{prop:meanders-to-NC-pairs}.

\begin{figure}[htbp]
\centering
\begin{minipage}{.5\textwidth}
  \centering
  \includegraphics[width=0.8\linewidth]{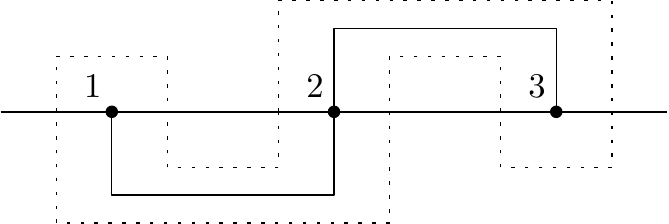}
  \captionof{figure}{A meander obtained from two non-crossing partitions $\{(1),(2,3)\}$ and $\{(1,2),(3)\}$.}
  \label{fig:permutation_eg}
\end{minipage}\begin{minipage}{.5\textwidth}
  \centering
  \includegraphics[width=0.8\linewidth]{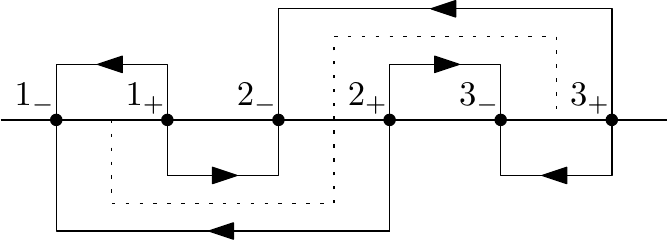}
  \captionof{figure}{Fattening with directions.}
  \label{fig:permutation2_eg}
\end{minipage}
\end{figure}

\medskip

We change now topics and discuss generating series associated to moments of probability measures and free cumulants. 
First, we define the moment generating function and R-transform:
$$ M(z) = \sum_{n \geq 1} m_n z^n \quad \text{and} \quad R(z) = \sum_{n \geq 1} \kappa_n z^n$$
where $m_n$ is the $n$-th moment and $\kappa_n$ is the $n$-th \emph{free cumulant}. 
Note that the free cumulants are defined by the following relation, called the \emph{moment-cumulant formula} \cite[Lecture 11]{nica2006lectures}:
\eq{\label{eq:m-k}
m_n = \sum_{\alpha \in NC(n)} \prod_{c\in \alpha} \kappa_{|c|}
}
Conversely, free cumulants can be expressed in terms of moments by using the M\"obius function on the $NC(n)$ lattice.  This implies that 
the two generating functions $M$ and $R$ are related by the following implicit equation (see \cite[Remark 16.18]{nica2006lectures}):
\begin{equation}\label{eq:fucntional-equation-R-M}
M(z) =  R(z(1+M(z))).
\end{equation}
Historically, the notion of R-transform is introduced by Voiculescu in a slightly different form \cite{voiculescu1985symmetries,voiculescu1986addition},
and can be defined for each compactly supported measure $\mu$ via the moments: $m_n = \int x^n \, d\mu(x)$. 

Next, we denote by $\mathcal F$ the transformation mapping an arbitrary power series $R$ to the unique power series $M$ 
which consist of sequences in \eqref{eq:m-k}, or equivalently
satisfy \eqref{eq:fucntional-equation-R-M}.
In this case, we write 
\eq{\label{eq:F-transform}
\mathcal F :  R \mapsto  M
}
The $\mathcal F$-transform is used in Section \ref{sec:M-I} where it plays a crucial role in the derivation of our main result. Its use in combinatorics predates its incarnation in free probability: in \cite{beissinger1985enumeration}, the author relates generating series for some classes of combinatorial objects to the generating series of \emph{irreducible objects} of the same type. Remarkably, the lattices studied in \cite{beissinger1985enumeration} are precisely the ones which appear in non-commutative probability theory: all partitions (in relation to classical, or tensor independence), non-crossing partitions (in relation to free independence) and ``interval-block'' partitions (in relation to Boolean independence). 

\section{From meanders to pairs of non-crossing partitions} \label{sec:M-NC}
The following result appears in several places in the literature \cite[Theorem 3.3]{hal06}, \cite[Theorem 5.7]{sav}, \cite[Section IV.C]{fsn}, or \cite[Section 3]{nica2016free}. We state it here using our language, where we identify non-crossing partitions and geodesic permutations. 

\begin{proposition}\label{prop:meanders-to-NC-pairs}
Meandric systems on $2n$ points with $n-r$ loops are in bijection with the set 
\begin{equation}\label{eq:M-nr}
M_{n,r} := \{ (\alpha,\beta) \in \mathrm{NC}(n)^2 \, : \, \|\alpha^{-1}\beta\| = r \}.
\end{equation}
We denote by $\mathcal M(\alpha, \beta)$ the  
meandric system associated to the pair $(\alpha, \beta)$.
\end{proposition}

Before proving the result, let us describe how the bijection works on an example for $n=5$. 
In Figure \ref{fig:MtoP}, two different geodesic permutations are represented by black lines in the upper and lower sides of the horizontal line: the permutation $\alpha$ is depicted on top, while $\beta$ is depicted below the horizontal line. 
First, let us focus on the upper permutation, which is $\alpha = (1,2,3)(4,5)$. 
The red lines are the fattening of the permutation which is the non-crossing paring: $\bar \alpha = (1_-,3_+)(1_+,2_-)(2_+,3_-)(4_-,5_+)(4_+,5_-)$.
The red arrows show how this non-crossing paring is related to the original permutation. 
Indeed, $1\mapsto 2 \mapsto 3 \mapsto 1$ is represented by $1_+ \mapsto 2_-$, $2_+ \mapsto 3_-$ and $3_+ \mapsto 1_-$.
Similarly, $4 \mapsto 5 \mapsto 4$ is indicated by $4_+ \mapsto 5_-$ and $5_+ \mapsto 4_-$.
Next, the black lines in lower part represent $\beta = (1)(3)(2,4,5)$, and 
the fattening $\bar \beta =(1_-,1_+)(2_-,5_+)(2_+,4_-)(3_-,3_+)(4_+,5_-) $ is drawn by blue lines. 
However, this time the arrows are reversed, i.e., they direct from $*_-$ to $*_+$ while $*_+$ to $*_-$ for the red lines,
where $* \in \{1,\ldots,5\}$.
This is because we want to consider $[(1)(3)(2,4,5)]^{-1}$ where $5 \mapsto 4 \mapsto 2 \mapsto 5$ is represented by
$5_- \mapsto 4_+$, $4_- \mapsto 2_+$ and $2_- \mapsto 5_+$. 
Then, joining red and blue lines, we get the loop structure of the corresponding meanders and 
the number of loops in the meanders equals the number of loops of $(1,2,3)(4,5)\circ(1)(3)(5,4,2) = (1,2,4,3)(5)$, which is $2$. 
In this example, we verified that the number of loops in the induced meandric system is $2=5-3$ 
where $\|\alpha\beta^{-1}\|=3$. 

\begin{figure}[htbp] 
  \centering
    \includegraphics[width=1\textwidth]{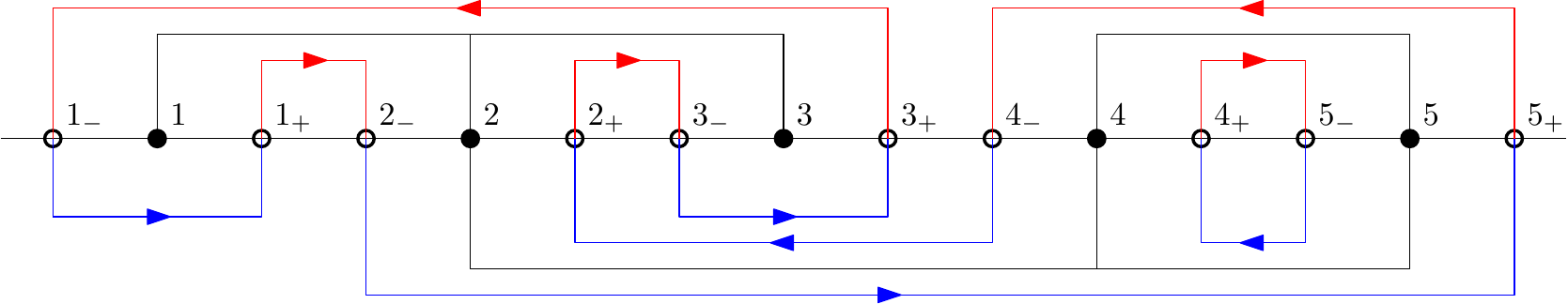}
\caption{Meanders generated by $\alpha = (1,2,3)(4,5)$ and $\beta = (1)(3)(2,4,5)$.}
 \label{fig:MtoP}
\end{figure}

\begin{proof}[Sketch of proof of Proposition \ref{prop:meanders-to-NC-pairs}]
We show the following bijection:
\eq{
NC(n) \times NC(n) = NC_2 (2n) \times NC_2(2n) = \{\text{meandric systems}\}  
}
where the first identification is the fattening;
$NC(n)$ on $\{i\}_{i=1}^n$ and $NC_2(2n)$ on $\{i_-,i_+\}_{i=1}^n$.
Namely, for two permutations $\alpha,\beta \in S_n$, 
placing $i_-$ on the left side of $i$ and $i_+$ on the right, we identify
\eqb{
 i &\mapsto \alpha(i)  \qquad\qquad& j &\mapsto \beta^{-1}(j)  \quad\text{respectively by} \\
i_+ &\mapsto \alpha(i)_- & j_- &\mapsto \beta^{-1}(j)_+ .
}
Here, the second line shows how the meandric system shapes as described in the above example. 
I.e., we can move along meanders to count the number of cycles:
\eq{
1_- \to [\beta^{-1} (1)]_+ \to [\alpha \beta^{-1}(1)] _- \to [\beta^{-1} \alpha \beta^{-1}(1) ]_+ \to \ldots \to 1_-  
}
As you can see, we visit $*_-$ and $*_+$ in turn ($* \in \{1,\ldots,n\}$), and then $( \alpha \beta^{-1})^m(1_-)=1_-$ for some $m \in \mathbb N$.
This sequence corresponds to a cycle in the permutation $\alpha \beta^{-1}$ which is a loop in the meandric system.
Then, we choose $i$ such that $i_-$ has not been visited, and identify another cycle in a similar way. 
We continue this process until we exhaust all $\{i_-\}_{i=1}^n$ to count the number of cycles in $\alpha \beta^{-1}$.
To finish the proof, remember that $\#(\alpha\beta^{-1})=n-\|\alpha\beta^{-1}\| = n-\|\alpha^{-1}\beta\|$.
\end{proof}

\medskip

To close this section we would like to overview how problems of meandric systems are related to other three interesting research topics, besides non-crossing partitions.  

Firstly, the meandric numbers $M_n^{(1)}$ count also the number of configurations of a folding closed polymer in 2D,
where a meander is regarded as a polymer. See Figure \ref{fig:polymer_eg}, which is compared with Figure \ref{fig:meander_eg}.
A folding open polymer can be thought of a \emph{semi-meander}, but we do not treat this problem in this current paper. 
Interested readers are referred to \cite{difrancesco1997meander}. 

\begin{figure}[htbp]
\begin{minipage}{.5\textwidth}
 \centering
  \includegraphics[width=0.5\linewidth]{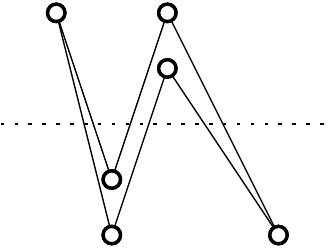}
  \captionof{figure}{A folding closed polymer.}
  \label{fig:polymer_eg}
\end{minipage}\begin{minipage}{.5\textwidth}
 \centering
  \includegraphics[width=0.8\linewidth]{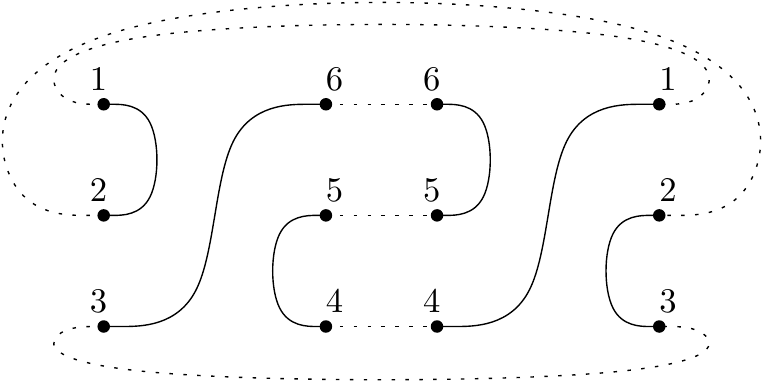}
  \captionof{figure}{$\trace [(e_1e_2)\cdot(e_2e_1)^T]$ in the Temperley-Lieb Algebra.}
  \label{fig:TLA_eg}
\end{minipage}
\end{figure}

Secondly, the scalar product in the Temperley-Lieb Algebra corresponds to
a meandric system, which contributes to $M_n^{(k)}$. 
This algebra is generated by $\{1,e_1,\ldots,e_{n-1}\}$ where 
\begin{enumerate}
\item $e_i^2 = w\cdot e_i$ for $i=1,\ldots,n-1$. 
\item $[e_i,e_j] = 0$ if $|i-j| \geq 2$. 
\item $e_i \cdot e_{i\pm1}\cdot e_i =e_i$ for $i=1,\ldots,n-1$. 
\end{enumerate} 
Here, $w>0$ is a scalar weight. 
The braid graphical representations of $e_i$  and $e_1\cdot e_2$ are found in Figure \ref{fig:braid_eg} and \ref{fig:braid2_eg}.
As you can see in Figure \ref{fig:TLA_eg}, $e_1\cdot e_2$ can be identified to the upper side of Figure \ref{fig:meander_eg}
and $(e_2e_1)^T $ to the lower side so that the scalar product is 
the weight of loop, say $w$, powered to the number of loops of the meandric system. 
Interested readers are referred to \cite{DGG_alg}.  

\begin{figure}[htbp]
\begin{minipage}{.5\textwidth}
  \centering
  \includegraphics[width=0.4\linewidth]{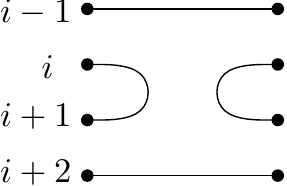}
  \captionof{figure}{$e_i$ in the braid representation.}
  \label{fig:braid_eg}
\end{minipage}\begin{minipage}{.5\textwidth}
 \centering
  \includegraphics[width=0.9\linewidth]{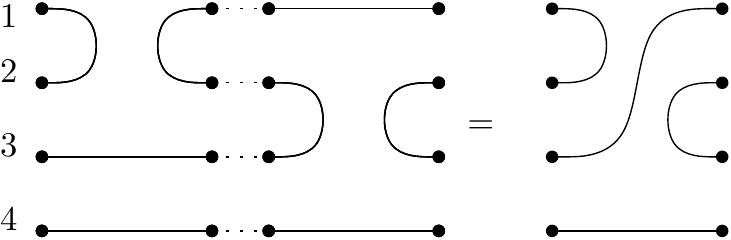}
  \captionof{figure}{$e_1 \cdot e_2$  in the braid representation.}
  \label{fig:braid2_eg}
\end{minipage}
\end{figure}

Thirdly, there is a mathematical relation to quantum information theory. 
While studying partial transpose of random quantum states  \cite{fsn}, 
we came across  \emph{meander polynomials} :
\begin{align}
M_n (x) := \sum_{k=1}^n M_n^{(k)} x^k
\end{align}
Based on this fact, 
a new random matrix model for the meander polynomials was found. 
That is,
if you take a complex Gaussian random matrix $G \in M_{N^2,x} (\mathbb C)$ with the mean $0$ and the variance $1/N$ for each entry,
we get it as the limiting moments:
\begin{align}
M_n (x) = \lim_{N \to \infty} \frac 1 {N^2} \trace \mathbb E \left( (GG^*)^\Gamma \right)^{2n} 
\end{align} 
where $\Gamma$ is called \emph{partial transpose}, with which we apply transpose only to one of spaces 
of the bipartite system $\mathbb C^n \otimes \mathbb C^n$. 
Note that $GG^*$ will be a random quantum state (i.e.~a positive semidefinite matrix with unit trace) with proper normalization.  

\section{Irreducible meandric systems}
\label{sec:irred}

The concept of \emph{irreducible meandric systems} was introduced by Lando and Zvonkin in \cite{lando1993plane}. 
Informally, these are meandric systems on $2n$ points such that 
there is no interval $[a,b] \subseteq [1,2n]$ with the property that the restriction of the meandric system to $[a,b]$ is another meandric system; 
see Figure \ref{fig:irred-vs-red} for examples. Let us explain here the elements appearing in Figure \ref{fig:irred-vs-red}, 
since such diagrams will continue to be used from here on. The meandric systems are represented by the red, thin lines; here, the red curve(s) intersect the horizontal line $2n=8$ times. The non-crossing partitions $\alpha$ and $\beta$ generating the meander are represented with black lines, above and respectively below the horizontal line; in the left panel of Figure \ref{fig:irred-vs-red} we have $\alpha = (24)$ and $\beta = (13)$. Finally, the $n=4$ points on which the geodesic permutations corresponding to $\alpha$ and $\beta$ act are represented by blue dots. 

\begin{figure}[htbp]
\begin{center}
\includegraphics[scale=0.2]{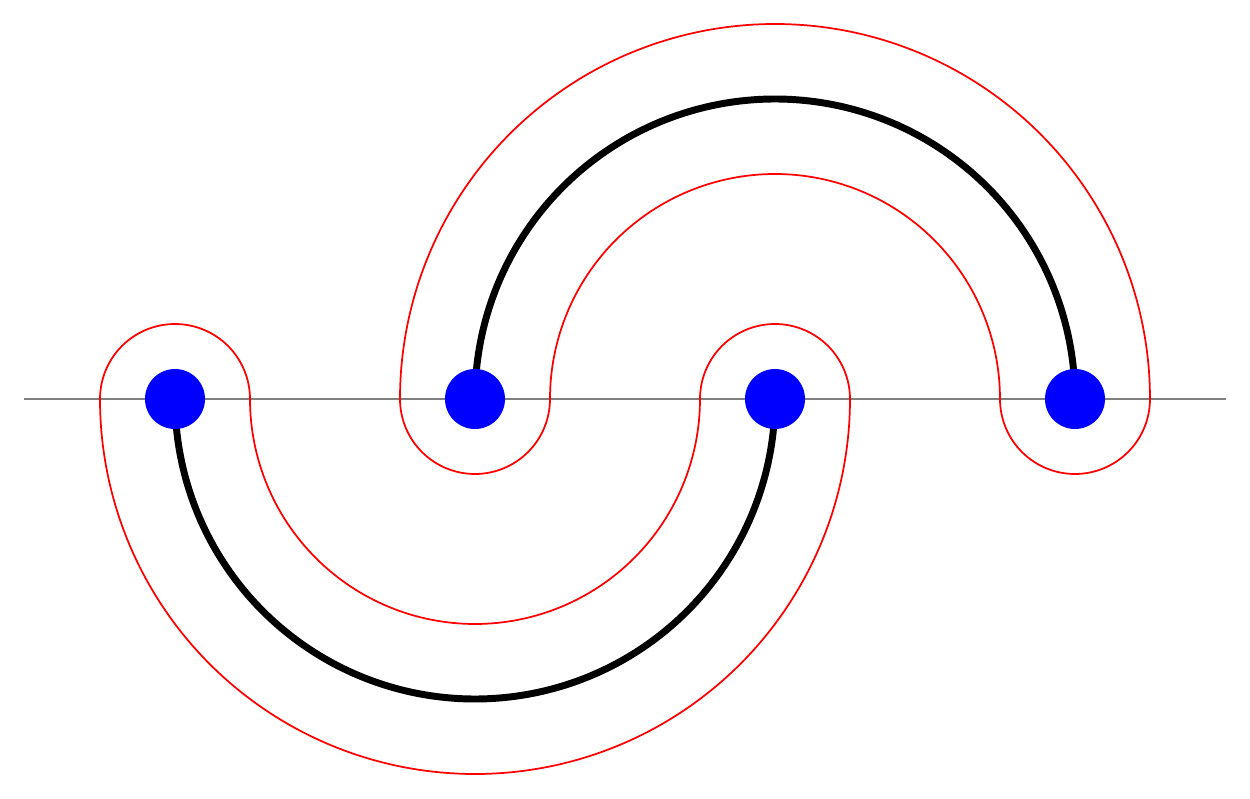} \qquad  \qquad  \qquad  
\includegraphics[scale=0.2]{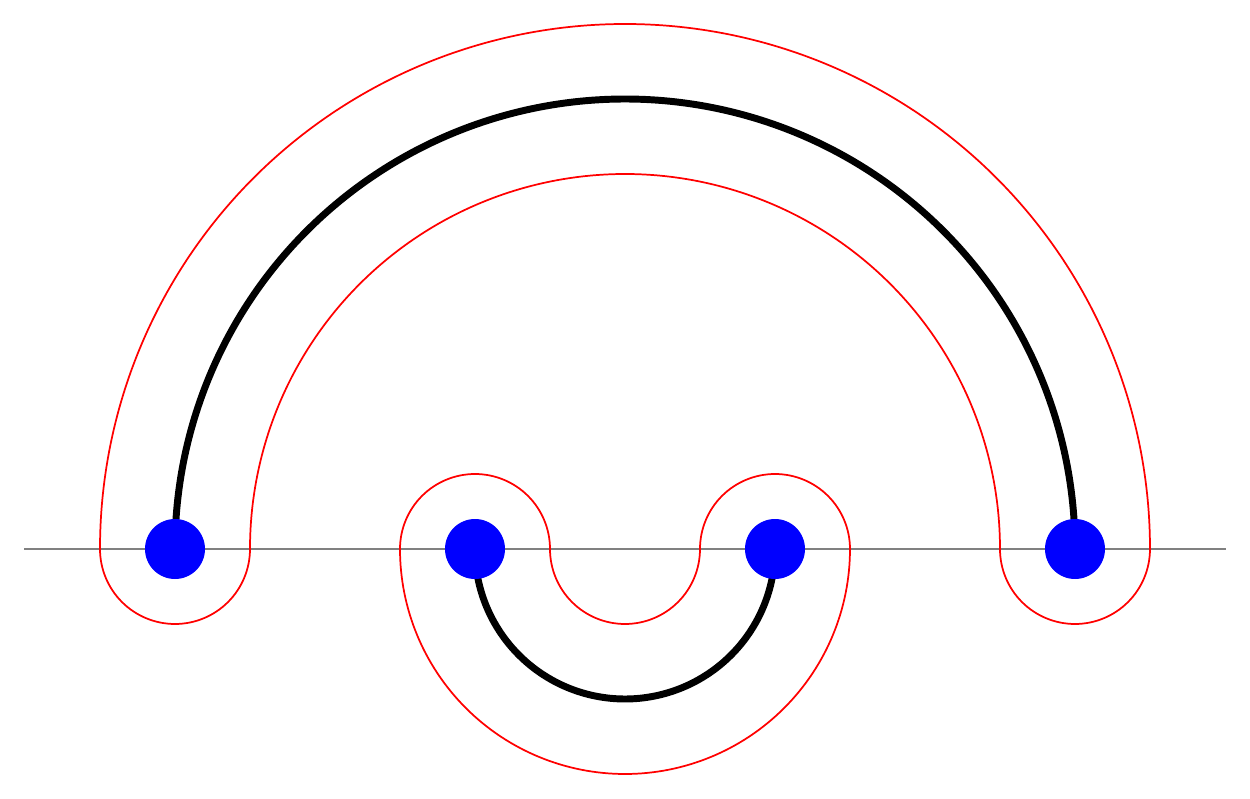}
\end{center}
\caption{On the left, an irreducible meandric system. On the right a reducible meandric system: the diagram in the center gives a proper meandric system.}
\label{fig:irred-vs-red}
\end{figure}

In the language of non-crossing partitions, irreducible meandric systems have been shown \cite[Theorem 1.1 or Proposition 3.4]{nica2016free} to be in bijection with the set
\begin{equation}\label{eq:def-In}
I_{n}:=\{(\alpha,\beta) \in NC(n)\, : \, \alpha \wedge \beta = 0_n,\, \alpha \vee \beta = 1_n\}.
\end{equation}

Interestingly, Lando and Zvonkin show that if $C(x)$ and, respectively, $I(x)$ are the generating series for the square Catalan numbers, and the number of irreducible meandric systems
$$C(x) = \sum_{n \geq 0} \mathrm{Cat}_n^2 x^n, \qquad I(x) = \sum_{n \geq 0} |I_n| x^n,$$
they satifsfy the functional equation $C(x) = I(xC^2(x))$, allowing them to obtain the asymptotic growth rate of the number of irreducible meandric systems \cite{lando1993plane}
$$\limsup_{n \to \infty} |I_n|^{1/n} = (\pi/(4-\pi))^2.$$
We note that the problem of estimating the asymptotic growth of the sequence $(M_n^{(1)})_n$ (the number of meanders, i.e.~meandric systems with one loop) is largely open. It is conjectured that
$$M_n^{(1)} \sim C \rho^n n^{-\kappa},$$
with $\kappa = (29+\sqrt{145})/12$ \cite{di2000meanders,di2000exact} and $\rho \approx 12.26287$, while it is known that $11.380 \leq \rho \leq 12.901$ \cite{albert2005bounds}. We do not discuss this problem here, and we think that tackling it would require some new ideas (see \cite[Section 5]{nica2016free} for some recent considerations). In this work, we do compute the asymptotic behavior of number of meandric systems on $2n$ points with $n-r$ loops (for fixed $r$), see Corollaries \ref{cor:asympt} and \ref{cor:asympt-precise-6}.

\medskip

One of the main new insights of the current work is to further partition the set of irreducible meandric systems in terms of the lengths of the permutations $\alpha,\beta$ and in terms of the distance between $\alpha$ and $\beta$.

\begin{definition}
We call a pair of non-crossing partitions $(\alpha, \beta) \in NC(n)^2$ \emph{irreducible of type $(n,r,a,b)$} if the following conditions are simultaneously satisfied:
\begin{enumerate}
\item $\alpha \wedge \beta = 0_n$
\item $\alpha \vee \beta = 1_n$
\item $\|\alpha^{-1}\beta\| = r$
\item $\|\alpha\|=a$
\item $\|\beta\|=b$.
\end{enumerate}
The corresponding meandric system $\mathcal M(\alpha, \beta)$ is also called irreducible of type $(r,a,b)$. We write
\begin{equation}\label{eq:def-I-nrab}
I_{n,r,a,b}:=\{(\alpha,\beta) \in NC(n)^2\, : \, \alpha \wedge \beta = 0_n,\, \alpha \vee \beta = 1_n, \, \|\alpha^{-1}\beta\| = r,\, \|\alpha\| = a,\, \|\beta\| = b\}.
\end{equation}
\end{definition}

Let us consider some examples. At $n=1$, we obtain the unique irreducible meandric system with $r=0$: $\alpha = \beta = (1)$; the corresponding triple of parameters is $(0,0,0)$. 
At $n=2$, we have the the following four possible parameter triples and the last two correspond to irreducible meandric systems (see Figure \ref{fig:n-2} for a graphical representation): 
\begin{itemize}
\item $(0,0,0)$: $\alpha = \beta = (1)(2)$
\item $(0,1,1)$: $\alpha = \beta = (12)$
\item $(1,0,1)$: $\alpha = (1)(2)$, $\beta = (12)$
\item $(1,1,0)$: $\alpha = (12)$, $\beta = (1)(2)$.
\end{itemize}

\begin{figure}[htbp]
\begin{center}
\includegraphics[scale=0.2]{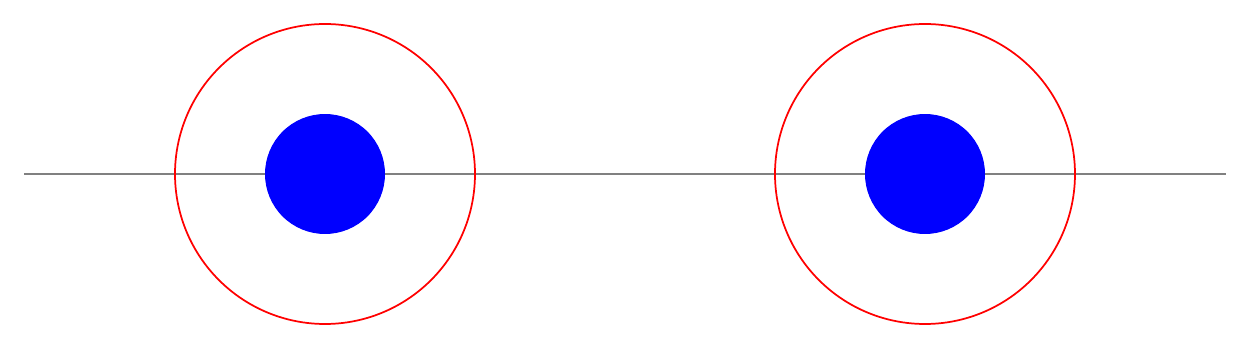} \qquad \includegraphics[scale=0.2]{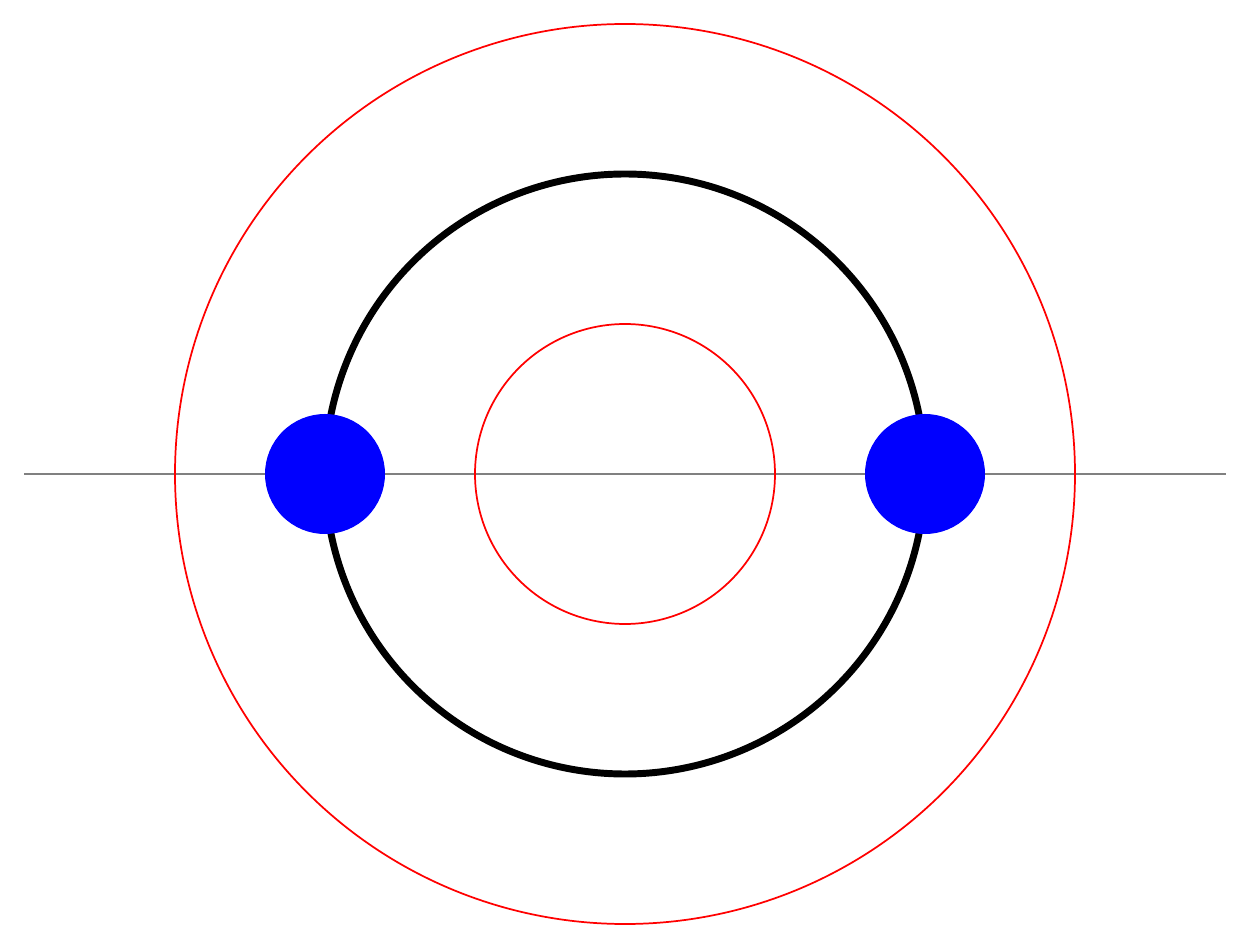} \qquad \includegraphics[scale=0.2]{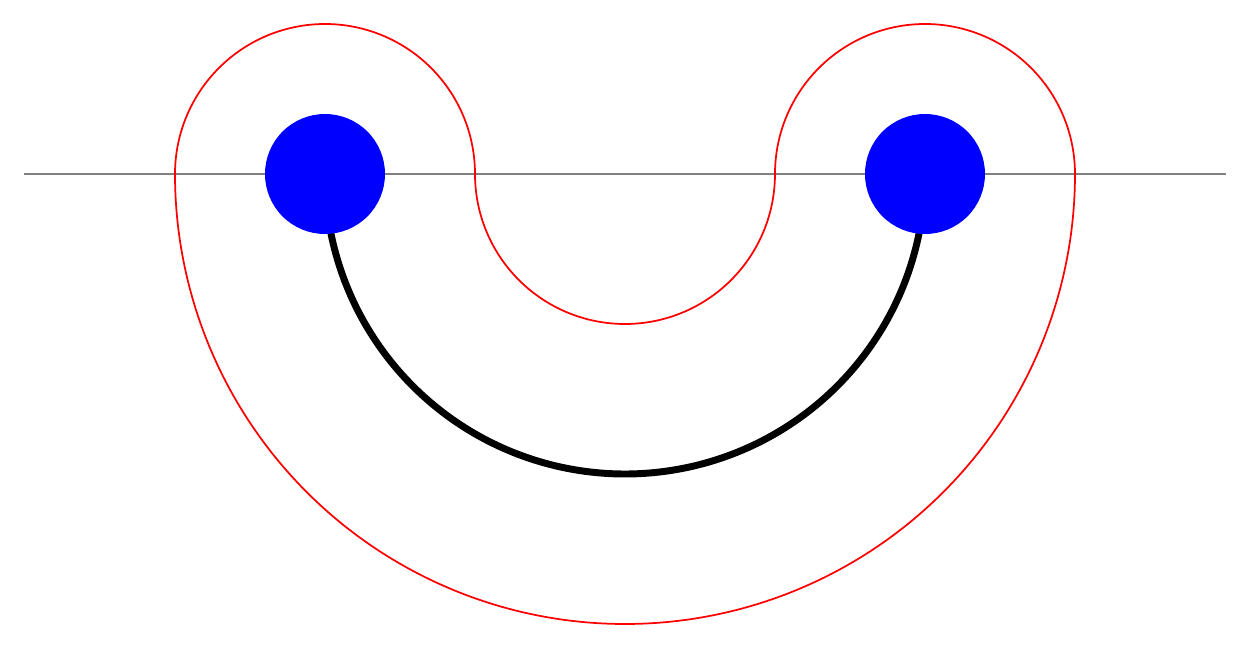} \qquad \includegraphics[scale=0.2]{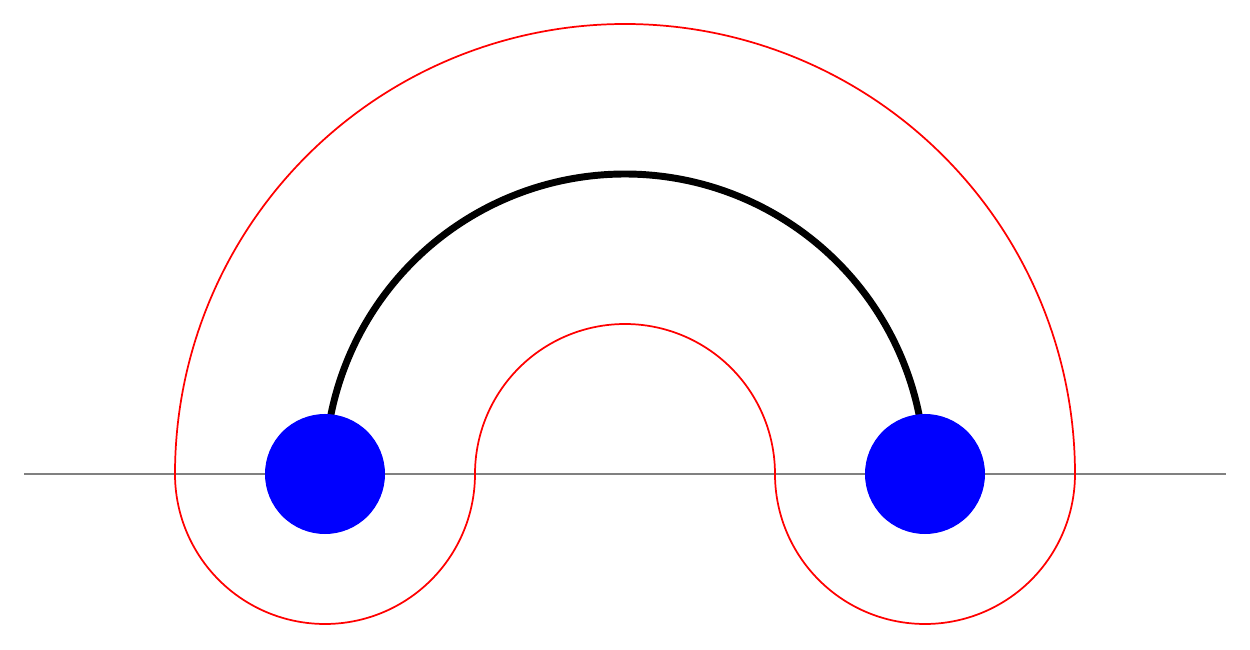}
\end{center}
\caption{All meandric systems on $n=2$ points: [$\alpha = \beta = (1)(2)$], [$\alpha = \beta = (12)$], [$\alpha = (1)(2)$, $\beta = (12)$], [$\alpha = (12)$, $\beta = (1)(2)$]. Only the last two examples correspond to irreducible meandric systems.}
\label{fig:n-2}
\end{figure}

One of the key facts that will be used in what follows is that the parameters $n,r,a,b$ need to verify some restrictions in order for the set $I_{n,r,a,b}$ to be non-empty. 

\begin{definition}\label{def:compatible}
A quadruple of non-negative integers $(n,r,a,b)$ is called \emph{compatible} if it satisfies the following conditions:
\begin{enumerate}
\item $a,b \leq \max(2r-2,1)$
\item $|a-b| \leq r \leq a+b$
\item $a-b$ and $a+b$ have the same parity as $r$
\item If $r=|a-b|$, then $\min(a,b) = 0$ and $\max(a,b)=r$
\item $r+1 \leq n \leq 2r + \mathbf{1}_{n=1}$.
\end{enumerate}
In particular, a triple $(r,a,b)$ is called compatible if it satisfies the first four conditions above for some fixed $n$. 
\end{definition}

As an example, in Figure \ref{fig:compatible}, we have indicated by filled disks the possible values of $a,b$ such that 
the quadruple $(n,r=6,a,b)$ is compatible, disregarding the value of $n \leq 12$. 
It turns out however that the compatibility conditions 
are not sufficient to ensure that $I_{n,r,a,b} \not =\emptyset$: 
for all $n\leq 12$, the sets $I_{n,6,8,10}$, $I_{n,6,9,9}$, $I_{n,6,10,8}$, and $I_{n,6,10,10}$ are empty.

\begin{figure}[htbp]
\begin{center}
\includegraphics{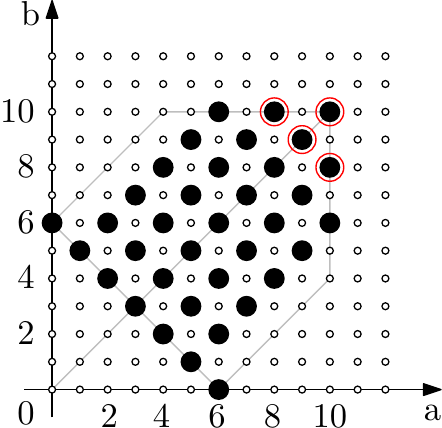}
\caption{The set of pairs $(a,b)$ such that the quadruple $(n,r=6,a,b)$ is compatible is depicted by large, filled disks. It turns out that some of these quadruples (the ones marked with a red circle) correspond to empty sets $I_{n,r,a,b}$: 
this implies that the compatibility conditions from Definition \ref{def:compatible} are not sufficient for $I_{n,r,a,b} \not =\emptyset$.}
\label{fig:compatible}
\end{center}
\end{figure}

We prove now the main result of this section.

\begin{proposition}\label{prop:compatible}
If a quadruple $(n,r,a,b)$ is not compatible, then the corresponding set $I_{n,r,a,b}$ is empty.
\end{proposition}
\begin{proof}
The claim (2) comes from the triangle inequality in the metric space $\mathcal S_n$ with the distance $d(\alpha, \beta) = \|\alpha^{-1}\beta\|$.
The claim (3) is clear because of well-definedness of parity; 
if one writes a permutation as products of transpositions, the parity is same for all possible products. 
To show (4), suppose $a \leq b$ without loss of generality. 
Then, $r= b-a$ corresponds to a tight case for the triangle inequality, and we have the geodesic: $\mathrm{id} - \alpha - \beta$,
so that $\alpha \wedge \beta = 0_n$ implies that $\alpha = 0_n$ and hence the claim. 
Next, we prove the upper-bound in (5).  The case $n=1$ corresponds to the case $(n,r,a,b)=(1,0,0,0)$. 
Suppose now $n \geq 2$ and we show that $n/2$ bounds from above the second term of the right hand side of the following identity:
$$ n=r + (\text{the number of loops in the meanders}).$$
To this end, note that for irreducible meandric systems  every loop intersects the horizontal line at least four times. 
Indeed, consider a meandric system of two non-crossing pairings $\bar \alpha, \bar \beta \in NC_2(2n)$ of $\{1_-,1_+,\ldots, n_-, n_+\}$. 
Suppose for a contradiction that there is a loop which intersects the horizontal line at only two points, 
for example, $\{i_-,j_+\}$ with $i \leq j$ or $\{i_+,j_-\}$ with $i<j$. 
By Lemma \ref{lemma:fat}, $i=j$ implies that $\alpha \vee \beta \not = 1_n$, and 
$i \not =j$ that $\alpha \wedge \beta \not = 0_n$. This contradiction proves that 
the number of loops in an irreducible meandric system is bounded by $\frac {2n}{4}$,
because it is impossible to draw a loop which intersects with the horizontal line an odd number of times. 
The lower bound in (5) follows from the fact that the diameter of $\mathcal S_n$ is $n-1$. 
Finally, we show (1) for $a$ (the proof for $b$ being similar). 
When $n\leq 2$, the claim is true based on the above observation. 
For $n \geq 3$ it is easy to see from (5) that 
\[
a \leq n-1 \leq 2r-1
\]
Suppose $a=n-1$ and 
this implies that $\alpha = 1_n$ and hence $\beta = 0_n$ for irreducible meandric systems. 
Then, we have $n=r+1 \leq 2r-1$ so that $a \leq 2r-2$.  
This completes the proof.
\end{proof}

The bound on $n$ in the result above is interesting: if we are interested in irreducible meandric systems with a fixed parameter $r$, we only need to investigate non-crossing partitions of sizes at most $2r$; this fact will be useful in the proof of our main result and also in the numerical procedures used to generate irreducible meandric systems \cite{num}. In Figure \ref{fig:all-irred-meanders-r-2}, we have represented all irreducible meandric systems with $r=2$; note that the maximal size of non-crossing partitions appearing in the list is $n=4$. 

\begin{figure}[htbp]
\begin{center}
\begin{tabular}{ccccccc}
\includegraphics[scale=0.2]{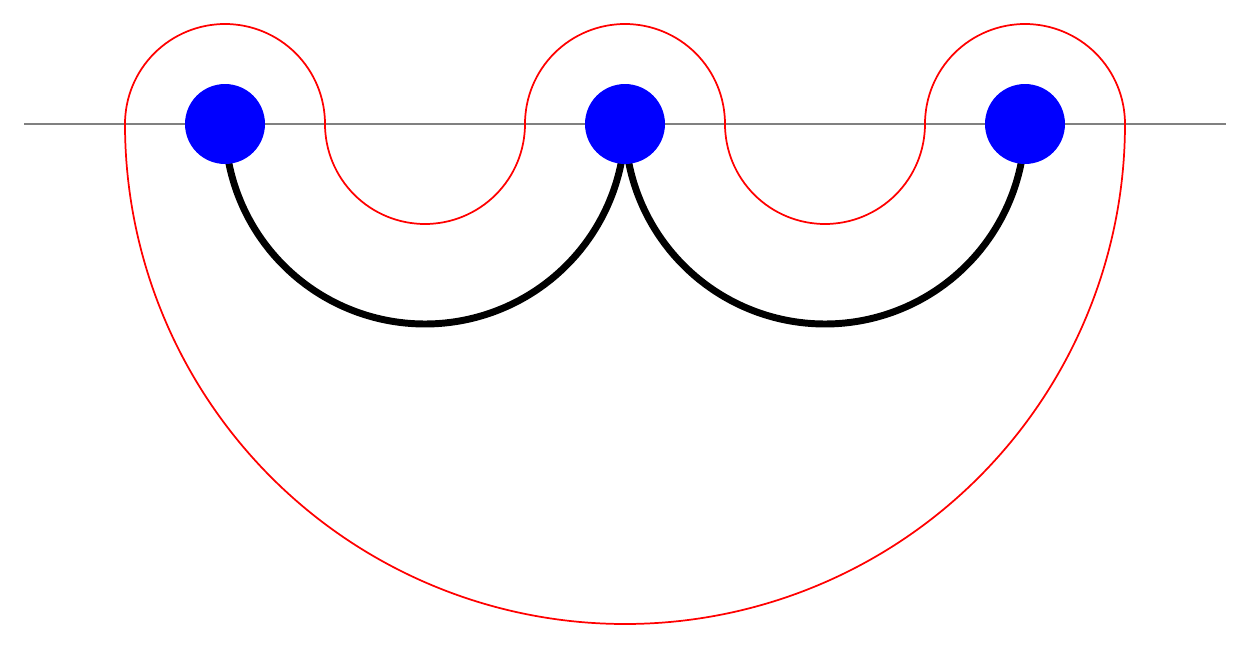} & \quad & 
\includegraphics[scale=0.2]{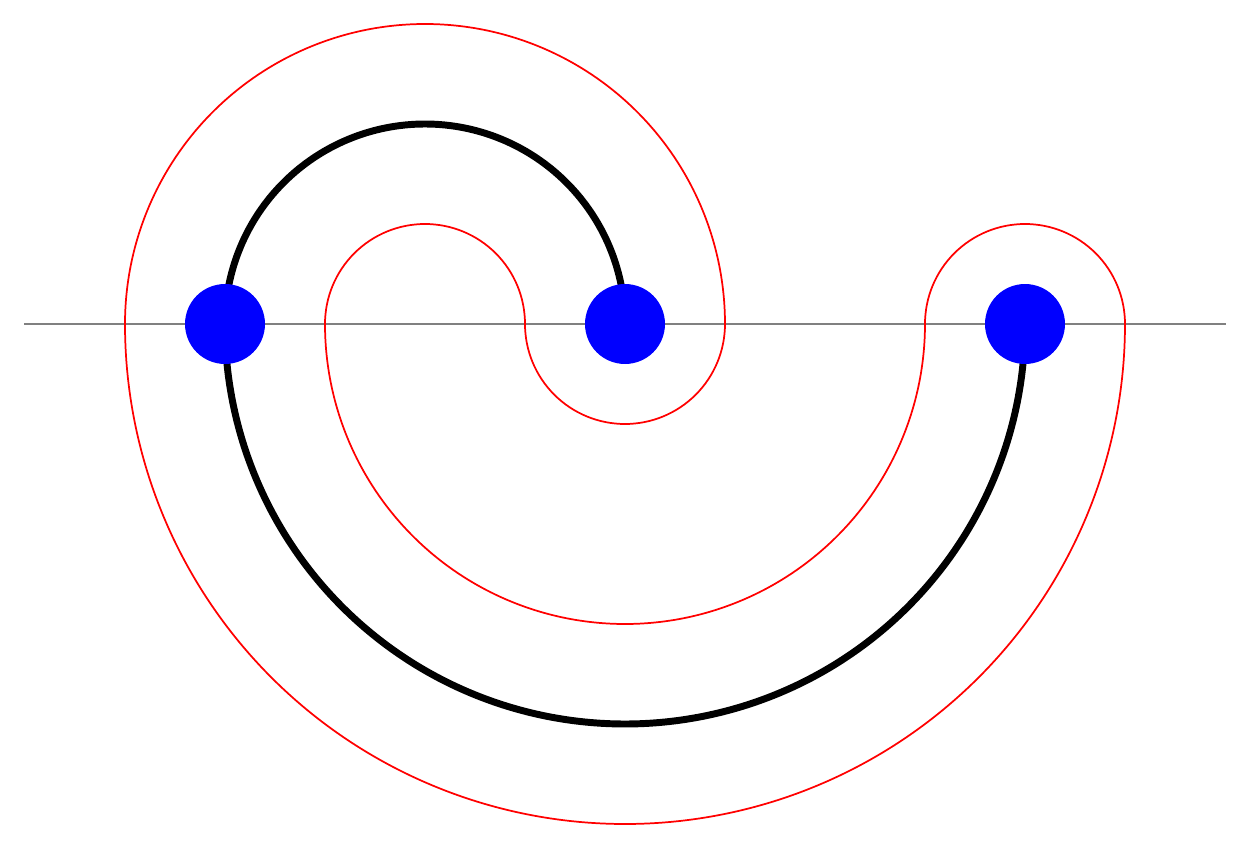}  & \quad & 
\includegraphics[scale=0.2]{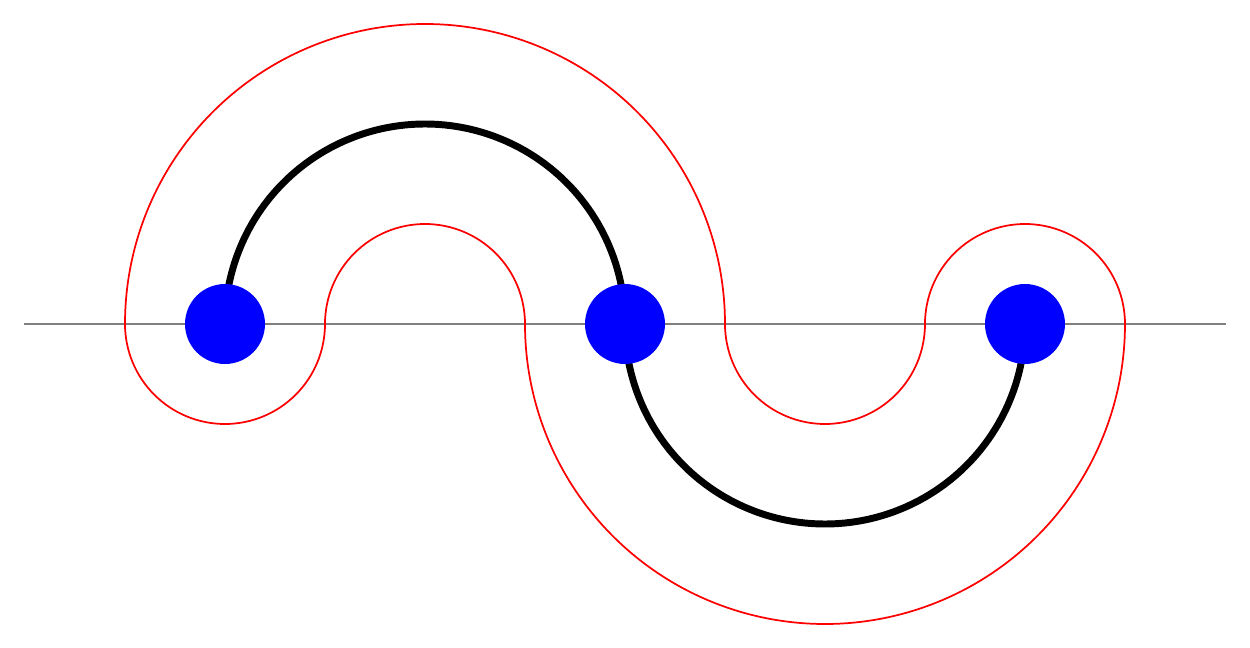} & \quad & 
\includegraphics[scale=0.2]{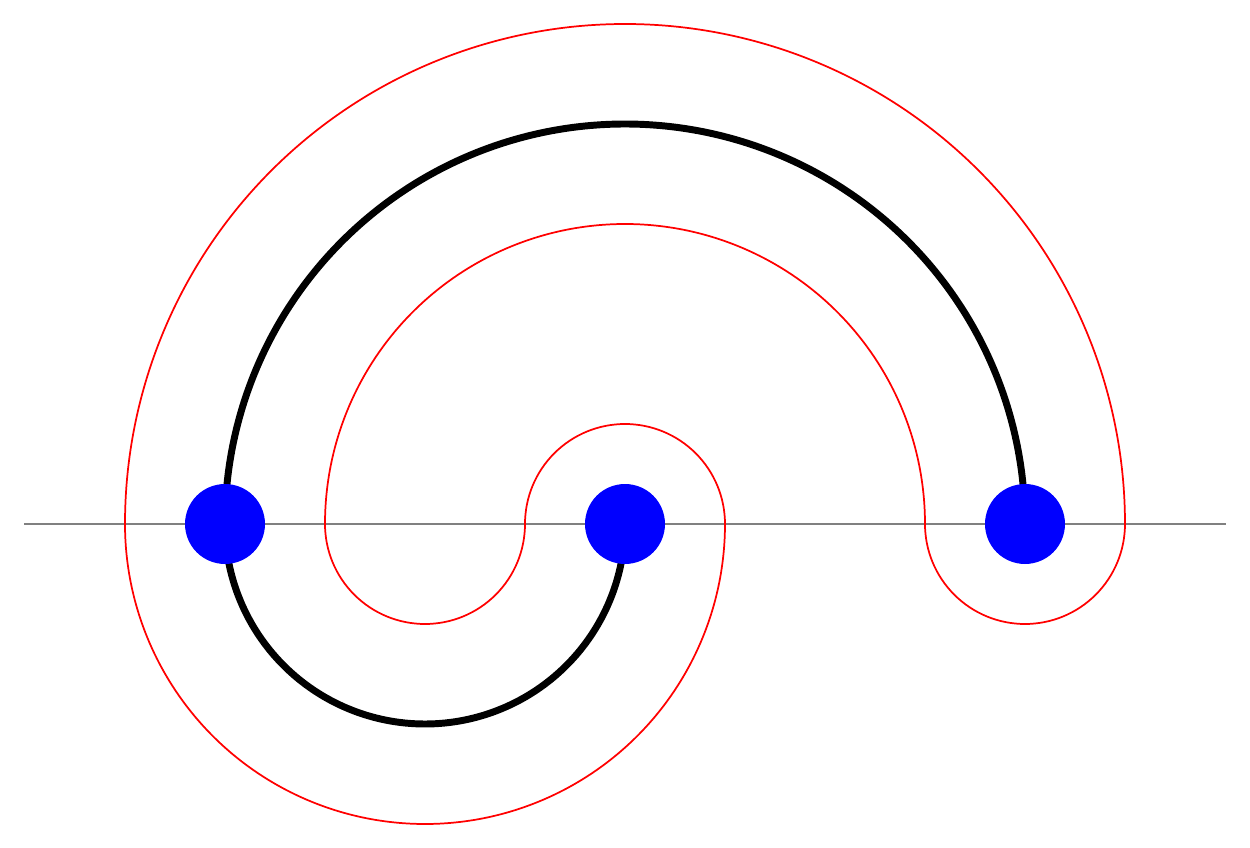} \\ \medskip 
 
\includegraphics[scale=0.2]{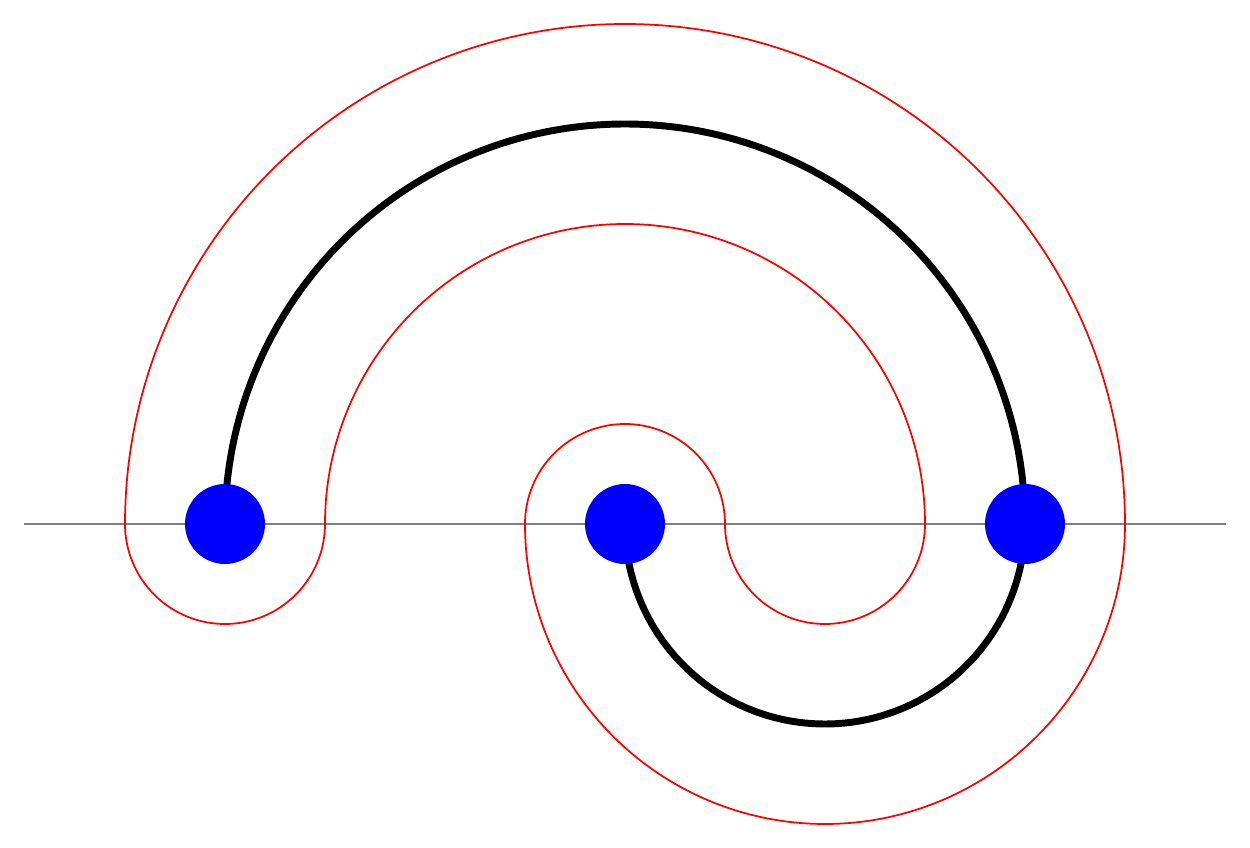} & \quad & \includegraphics[scale=0.2]{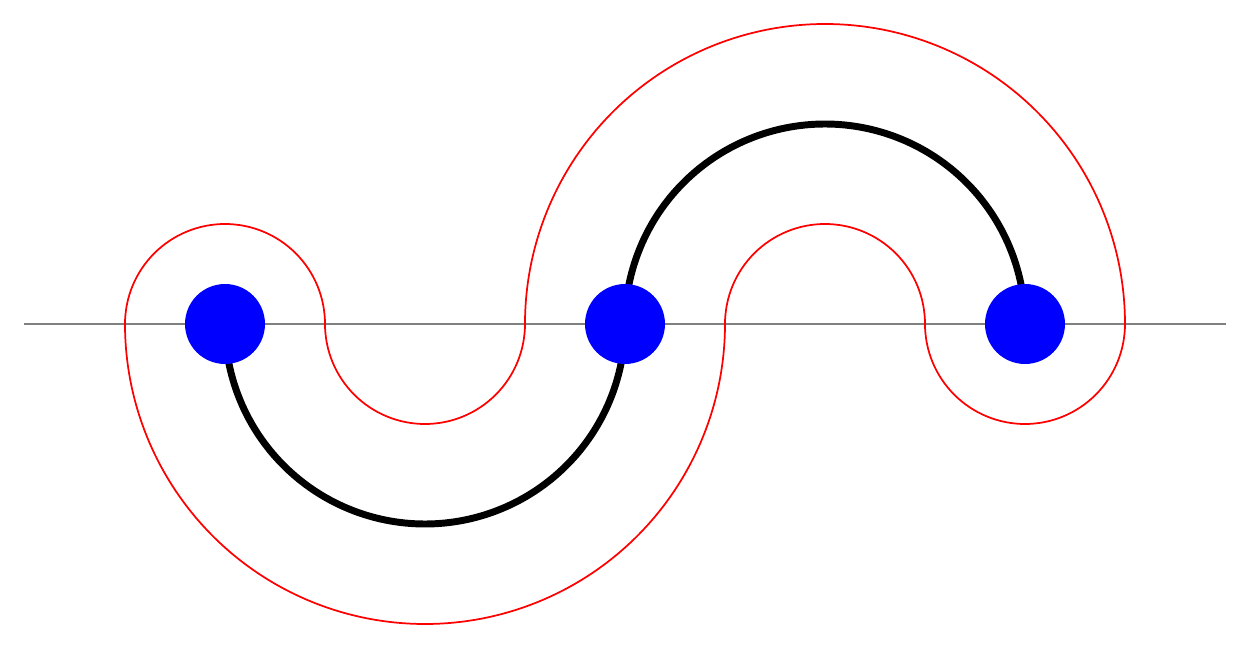} & \quad & 
\includegraphics[scale=0.2]{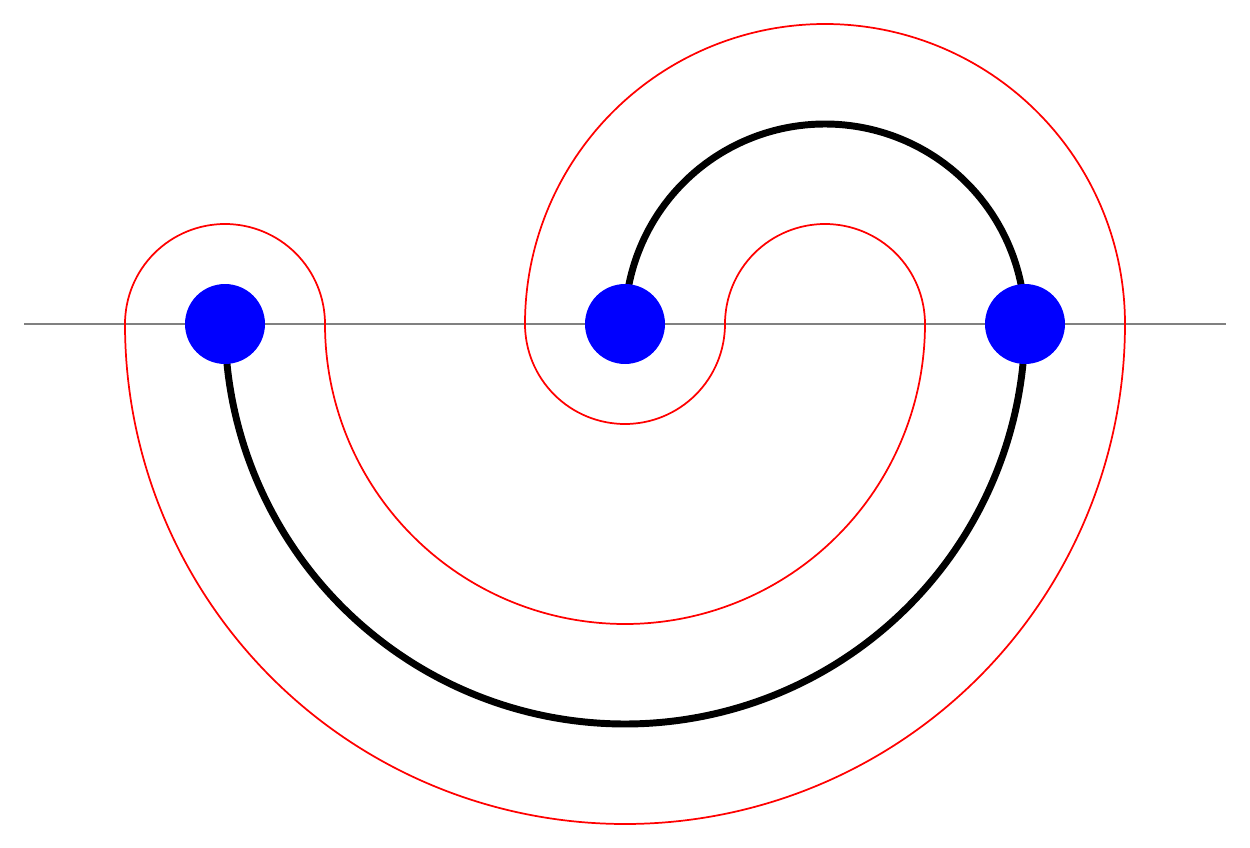} & \quad & 
\includegraphics[scale=0.2]{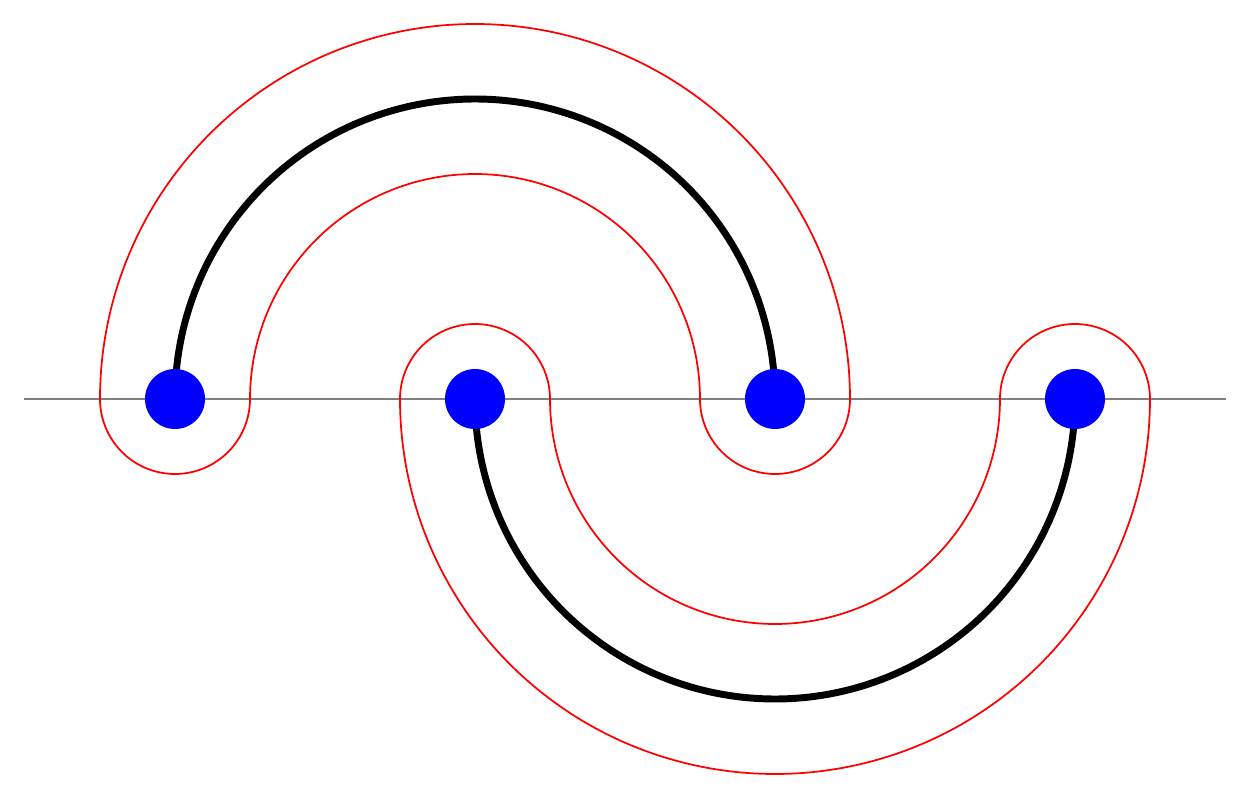} \\ \medskip 

\includegraphics[scale=0.2]{r-2-a-1-b-1-nb-14}  & \quad & 
\includegraphics[scale=0.2]{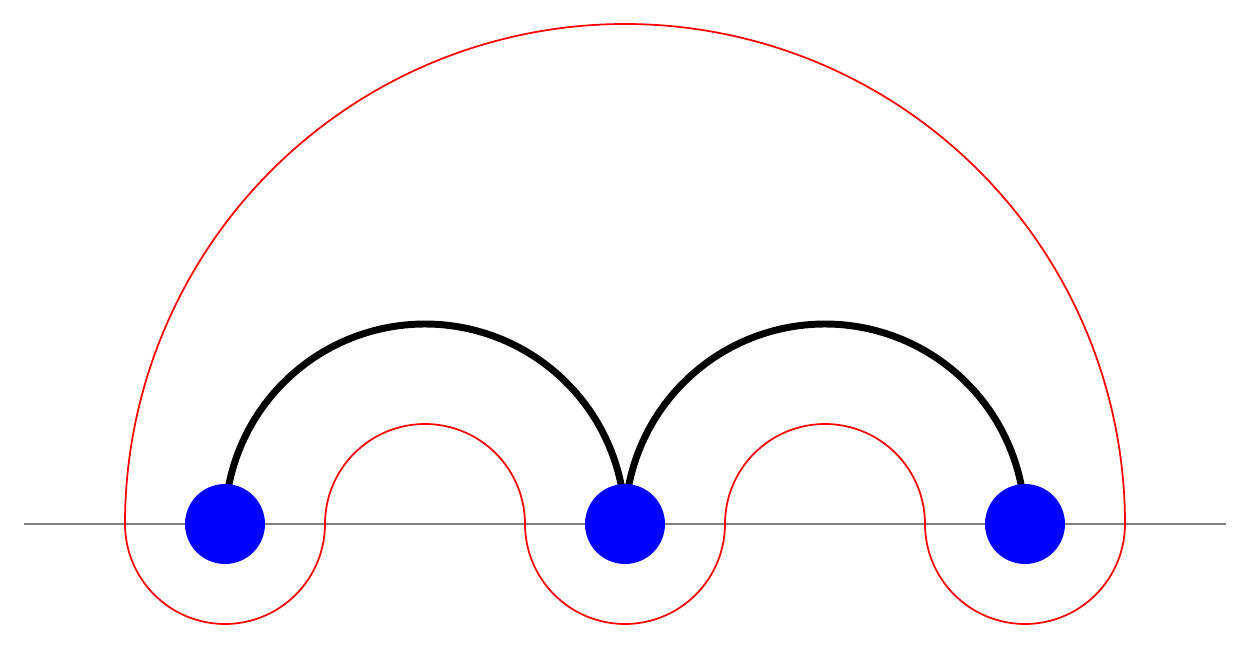} & \quad & 
\includegraphics[scale=0.2]{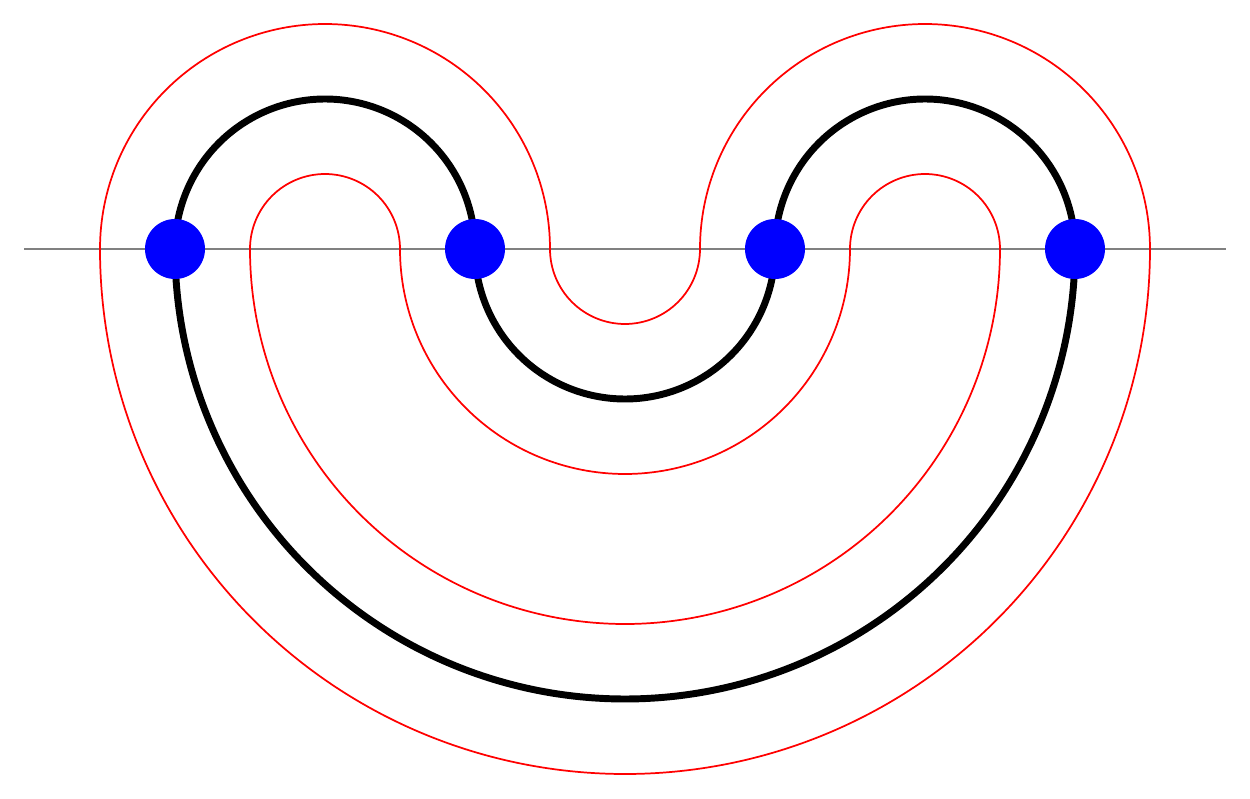} & \quad & 
\includegraphics[scale=0.2]{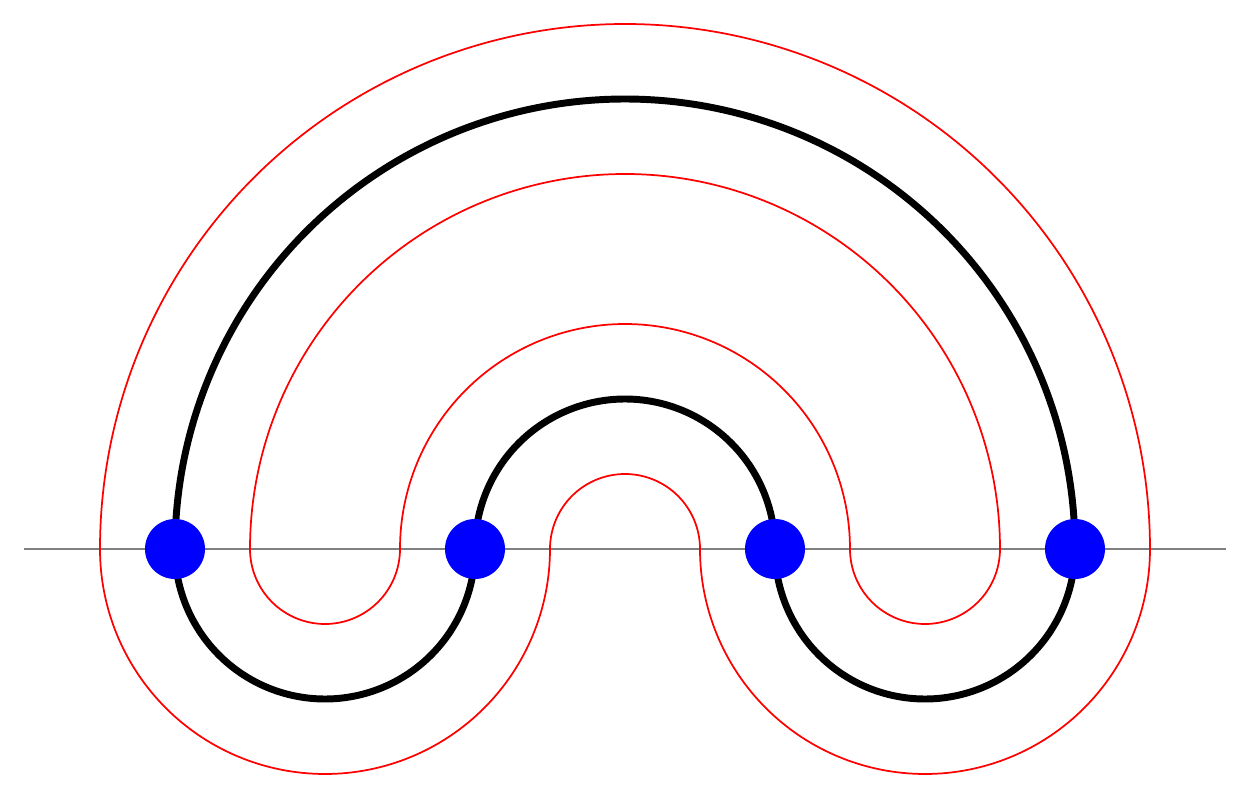}
\end{tabular}
\end{center}
\caption{All irreducible meandric systems of type $(2,*,*)$. From top to bottom, left to right, the first meander is of type $(2,0,2)$, the next 8 are of type $(2;1,1)$, the next is of of type $(2,2,0)$, and the last two are of type $(2,2,2)$. }
\label{fig:all-irred-meanders-r-2}
\end{figure}

\section{Counting meandric systems using irreducible meandric systems}
\label{sec:M-I}

We list  three important sets of meandric systems, i.e. pair of non-crossing partitions in $NC(n)^2$:
\eq{
\label{eq:I-nrab}
I_{n,r,a,b} &:= \{(\alpha,\beta) \, :\, \alpha \wedge \beta = 0, \, \alpha \vee \beta = 1, \,  \|\alpha^{-1}\beta\| = r,\, \|\alpha\| = a,\, \|\beta\| = b \}\\
\label{eq:K-nrab}
K_{n,r,a,b}&:=\{(\alpha,\beta) \, : \, \alpha \vee \beta = 1, \,  \|\alpha^{-1}\beta\| = r,\, \|\alpha\| = n-1-a,\, \|\beta\| = n-1-b \}\\
\label{eq:M-nrab}
M_{n,r,a,b}&:=\{(\alpha,\beta) \, : \, \|\alpha^{-1}\beta\| = r,\, \|\alpha^{-1}(\alpha \vee \beta)\| = a,\, \|\beta^{-1}(\alpha \vee \beta)\| = b\}
}
The three sets above count pairs of non-crossing partitions (or, equivalently, geodesic permutations), according to the $r,a,b$ statistics, having some particular geometric properties in the lattice $NC(n)$. The first one was already introduced in \eqref{eq:def-I-nrab} and its formal generating series for irreducible meandric systems, 
together with their statistics $r,a,b$, is 
\eqb{\label{eq:I}
I(X,Y,A,B) &= \sum_{(n,r,a,b) \text{ compatible}}  |I_{n,r,a,b}| X^nY^rA^aB^b\\
&=  \sum_{r,a,b \geq 0}   e_{r,a,b}(X) \cdot Y^rA^aB^b 
\quad \text{where}\quad e_{r,a,b}(X) = \sum_{n \geq 1}|I_{n,r,a,b}|X^n \\
 &=\sum_{n \geq 1}   e_{n}(Y,A,B) \cdot X^n 
\quad \text{where}\quad e_{n}(Y,A,B) = \sum_{r,a,b \geq 0}|I_{n,r,a,b}| Y^rA^aB^b 
}
This series starts as follows
\begin{align*}
I(X,Y,A,B) &=X +  X^2 YA +  X^2 Y B+  \\
&\qquad +X^3 Y^2 A^2 + 6 X^3 Y^2 A B  + X^3 Y^2 B^2 +  2X^4 Y^2  A B + 2 X^4 Y^2 A^2 B^2 + o(Y^2).
\end{align*}
Note that the coefficients of $Y^2$ (corresponding to $r=2$) are associated to the irreducible meandric systems from Figure \ref{fig:all-irred-meanders-r-2}.

In a similar fashion, we introduce a refinement of the set $M_{n,r}$ from \eqref{eq:M-nr}, to take into account the statistics $a,b$,
and we denote by $M$ its formal generating series
\eqb{\label{eq:M}
M(X,Y,A,B) &= \sum_{(n,r,a,b) \text{ compatible}}  |M_{n,r,a,b}| X^nY^rA^aB^b\\
&=  \sum_{r,a,b \geq 0}   g_{r,a,b}(X) \cdot Y^rA^aB^b 
\quad \text{where}\quad g_{r,a,b}(X) = \sum_{n \geq 1}|M_{n,r,a,b}|X^n \\
 &=\sum_{n \geq 1}   g_{n}(Y,A,B) \cdot X^n 
\quad \text{where}\quad g_{n}(Y,A,B) = \sum_{r,a,b \geq 0}|M_{n,r,a,b}| Y^rA^aB^b 
}
where  $M_{n,r,a,b}$ is defined in \eqref{eq:M-nrab}.
Note that for all $r \geq 0$, the generating function $F_r$ of meandric systems on $2n$ points with $n-r$ loops is given by
$$F_r(X) := \sum_{n \geq 1} M_n^{(n-r)} X^n = [Y^r]M(X,Y,1,1).$$
For this reason, our final goal is to collect information on generating function $M(X,Y,A,B)$.

Finally we introduce the formal generating series for $K_{n,r,a,b}$ defined in \eqref{eq:K-nrab}:
\eqb{\label{eq:K}
K(X,Y,A,B) &= \sum_{(n,r,a,b) \text{ compatible}}  |K_{n,r,a,b}| X^nY^rA^aB^b\\
&=  \sum_{r,a,b \geq 0}   f_{r,a,b}(X) \cdot Y^rA^aB^b 
\quad \text{where}\quad f_{r,a,b}(X) = \sum_{n \geq 1}|K_{n,r,a,b}|X^n \\
 &=\sum_{n \geq 1}   f_{n}(Y,A,B) \cdot X^n 
\quad \text{where}\quad f_{n}(Y,A,B) = \sum_{r,a,b \geq 0}|K_{n,r,a,b}| Y^rA^aB^b 
}
which is an intermediate definition bridging $I(X,Y,A,B)$ and $M(X,Y,A,B)$ .

We prove now the main result of this paper, connecting the two formal power series $I$ and $M$. 

\begin{theorem}\label{thm:M-I}
The formal generating series for $I$, $K$ and $M$ from \eqref{eq:I}, \eqref{eq:K} and \eqref{eq:M}, respectively, 
are related by the following relation:
\eq{
I \stackrel{\mathcal F}\longmapsto K \stackrel{\mathcal F}\longmapsto M
}
or equivalently, 
\eq{
\{e_{n}(Y,A,B)\}_{n \geq 1} \stackrel{\mathcal F}\longmapsto \{f_{n}(Y,A,B)\}_{n \geq 1}
\stackrel{\mathcal F}\longmapsto\{g_{n}(Y,A,B)\}_{n \geq 1}
}
for all $r,a,b \geq 0$. Remember that the transform $\mathcal F$ is defined in \eqref{eq:F-transform}.
\end{theorem} 
\begin{proof}
First, we have
\begin{align*}
g_{n}(Y,A,B) &
=  \sum_{\alpha,\beta \in NC(n)} Y^{\|\alpha^{-1}\beta\|} A^{\|\alpha^{-1}(\alpha \vee \beta)\|} B^{\|\beta^{-1}(\alpha \vee \beta)\|}\\
&=\sum_{\sigma \in NC(n)} \sum_{\substack{\alpha,\beta \in NC(n)\\\alpha \vee \beta = \sigma}} Y^{\|\alpha^{-1}\beta\|} A^{\|\alpha^{-1}\sigma\|} B^{\|\beta^{-1}\sigma\|}\\
&=\sum_{\sigma \in NC(n)} \prod_{c \in \sigma}
\sum_{\substack{\alpha,\beta \in NC(|c|)\\\alpha \vee \beta = 1_{|c|}}} Y^{\|\alpha^{-1}\beta\|} A^{\|\alpha^{-1}1_{|c|}\|} B^{\|\beta^{-1}1_{|c|}\|}
= \sum_{\sigma \in NC(n)} \prod_{c \in \sigma} f_{|c|} (Y,A,B)
\end{align*}
where $c \in \sigma$ is a cycle in $\sigma$ and $|c|$ is its length.  
The main idea here is that, for fixed $\sigma$, 
the functions $\|\alpha^{-1}\beta\|$, $\|\alpha^{-1}\sigma\|$ and $\|\beta^{-1}\sigma\|$ are multiplicative with respect to the cycles of $\sigma$
if $\alpha \vee \beta = \sigma$.
We have proved that  $\mathcal F: K \mapsto M$.

Second, we can apply a similar calculation to show $\mathcal F: I \mapsto K$, but this time we take the Kreweras complement:
\begin{align*}
f_{n}(Y,A,B) &
=  \sum_{\substack{\alpha,\beta \in NC(n)\\\alpha \vee \beta = 1_n}} 
Y^{\|\alpha^{-1}\beta\|} A^{\|\alpha^{-1} 1_n\|} B^{\|\beta^{-1}1_n\|}
=  \sum_{\substack{\alpha,\beta \in NC(n)\\\alpha \wedge \beta = 0_n}} 
Y^{\|\alpha^{-1}\beta\|} A^{\|\alpha\|} B^{\|\beta\|}\\
&=\sum_{\sigma \in NC(n)} \sum_{\substack{\alpha,\beta \in NC(n)\\ \alpha \wedge \beta = 0_n \\ \alpha \vee \beta = \sigma}} Y^{\|\alpha^{-1}\beta\|} A^{\|\alpha\|} B^{\|\beta\|}\\
&=\sum_{\sigma \in NC(n)} \prod_{c \in \sigma}
\sum_{\substack{\alpha,\beta \in NC(|c|)\\ \alpha \wedge \beta = 0_{|c|} \\  \alpha \vee \beta = 1_{|c|}}} Y^{\|\alpha^{-1}\beta\|} A^{\|\alpha\|} B^{\|\beta\|}
= \sum_{\sigma \in NC(n)} \prod_{c \in \sigma} e_{|c|} (Y,A,B)
\end{align*}
This completes the proof. 
\end{proof}

\begin{corollary} \label{cor:I-K-M}
We have
\begin{align}
K(X,Y,A,B) & = I (X(1+K(X,Y,A,B) ),Y,A,B)\\
M(X,Y,A,B) &= K (X(1+M(X,Y,A,B) ),Y,A,B)
\end{align}
\end{corollary}

Before starting evaluating the functions $f_{r,a,b}$ and $g_{r,a,b}$ we introduce some notation:
for a polynomial $P(Y,A,B)$, 
\begin{itemize}
\item $[Y^rA^aB^b]P$ is the coefficient of $Y^rA^aB^b$ in $P(Y,A,B)$.
\item $\displaystyle D[r,a,b]P =  \frac{\partial^{r+a+b}}{\partial Y^r \partial A^a  \partial B^b} P(Y,A,B)$
\item $\displaystyle  D_0[r,a,b]P =  \left.\frac{\partial^{r+a+b}}{\partial Y^r \partial A^a  \partial B^b} \right|_{Y=A=B=0}P(Y,A,B)$
\end{itemize}
This means that 
\eq{
 D_0[r,a,b]P = r!a!b! \cdot [Y^rA^aB^b]P.
}

\begin{proposition}\label{prop:coefficients-K}
The coefficients of the formal power series $K(X,Y,A,B)$ are as follows:
\begin{itemize}
\item For $r=0$ we have 
\begin{align} \label{eq:coeff-K-0}
f_{0,0,0}(X)= \frac{X}{1-X}
\end{align}
 and $f_{0,a,b}(X)=0$ for $(a,b) \not = (0,0)$.
\item More generally,  when $(r,a,b)$ is compatible, we have
\begin{equation}\label{eq:coeff-K}
f_{r,a,b}(X) = \frac{X^{r+1}Q_{r,a,b}(X)}{(1-X)^{2r+1}},
\end{equation}
where $Q_{r,a,b}(X)$ is a polynomial of degree at most $r-\mathbf{1}_{r\geq 1}$, with integer coefficients. 
\end{itemize}
\end{proposition}
\begin{proof}
Notice that $K_{n,0,a,b}$ is empty unless $(a,b)=(0,0)$, so that
$|K_{n,0,0,0}| =1$ for all $n \geq 1$.
Hence
\[
f_{0,0,0}(X) = X+ X^2 + X^3+ \ldots 
\]
which proves the first statement. 

For the induction step, we introduce the order relation $<$ on compatible triples naturally;
$(r_1,a_1,b_1) <(r_2,a_2,b_2)$ if $r_1 \leq r_2$, $a_1 \leq a_2$, $b_1 \leq b_2$ and $(r_1,a_1,b_1) \neq (r_2,a_2,b_2)$. 
We always assume that those triples are non-negative: $(r,a,b) \geq (0,0,0)$.
Let us assume now the conclusion holds for all triples $(r,a,b)<(r_0,a_0,b_0)$;
we only think of the case $r_0 \geq 1$ because of the compatibility condition.
We write $f=f_{r_0,a_0,b_0}$ and 
$K = K(X,Y,A,B)$ below.  Since 
\[
\sum_{n \geq 1} |I_{n,0,0,0}| (X(1+K))^n Y^rA^aB^b = X+XK
\]
we have by using Corollary \ref{cor:I-K-M}
\begin{align}
f(X) &= [Y^{r_0}A^{a_0}B^{b_0}] \left((X+XK) + \sum_{\substack{(r,a,b) \not =(0,0,0)}} \,
\sum_{n \geq 1} |I_{n,r,a,b}| \left( (X(1+K))^n Y^rA^aB^b \right) \right)\notag\\
&= Xf(X) + \sum_{\substack{(0,0,0) \\< (r,a,b) \leq\\ (r_0,a_0,b_0)}} \sum_{n=r+1}^{2r} |I_{n,r,a,b}| X^n [Y^{r_0-r}A^{a_0-a}B^{b_0-b}](1+K)^n,
\label{eq:find-f1}
\end{align}
where we have used the fact that $I_{n,r,a,b}$ is empty, unless $r+1 \leq n \leq 2r$, see Proposition \ref{prop:compatible}. 
Using the recurrence hypothesis we have
\begin{align}
[Y^{r_0-r}A^{a_0-a}B^{b_0-b}](1+K)^n 
&= \sum_{\substack{\sum_{i=1}^n (r_i,a_i,b_i) \\ =(r_0-r,a_0-a,b_0-b) } }
\prod_{i=1}^n \frac{X^{s(r_i)}Q_{r_i,a_i,b_i}(X)}{(1-X)^{2r_i+1}}\notag \\
&=  \frac{X^{r_0-r}}{(1-X)^{2(r_0-r)+n}} \cdot \hat Q_{r,a,b}(X)
\label{eq:find-f2}
\end{align}
because $(r_i,a_i,b_i) < (r_0,a_0,b_0)$ holds by the condition $(r,a,b) \not = (0,0,0)$.
Here, $s(r) = r+\mathbf{1}_{r \geq 1}$, where the indicator function comes from the fact that  $[Y^0A^0B^0](1+K) = 1/(1-X)$. 
Note that 
\eq{
\hat Q_{r,a,b} = \sum_{\substack{\sum_{i=1}^n (r_i,a_i,b_i) \\ =(r_0-r,a_0-a,b_0-b) } }\prod_{i=1}^n
X^{\mathbf{1}_{r_i \geq 1}}  \cdot Q_{r_i,a_i,b_i}
}
is a polynomial and moreover  
\begin{align}
\mathrm{deg} (\hat Q_{r,a,b}) \leq \sum_{i=1}^n   [\mathbf{1}_{r_i \geq 1} + r_i - \mathbf{1}_{r_i \geq 1}] = r_0 -r
\end{align}
The compatibility condition may give $0$ in \eqref{eq:find-f2}, but we do not treat such cases separately, because it does not make any difference. Hence, putting \eqref{eq:find-f1} and \eqref{eq:find-f2} together we have
\begin{align}
(1-X)^{2r_0+1} \cdot f(X) 
&=   \sum_{\substack{(0,0,0) \\< (r,a,b) \leq\\ (r_0,a_0,b_0)}} \sum_{n=r+1}^{2r} |I_{n,r,a,b}| X^n \cdot X^{r_0-r} (1-X)^{2r-n}\cdot \hat Q_{r,a,b} \\
&=  X^{r_0+1} \sum_{\substack{(0,0,0) \\< (r,a,b) \leq\\ (r_0,a_0,b_0)}} \sum_{n=r+1}^{2r} |I_{n,r,a,b}|  X^{n-r-1} (1-X)^{2r-n}\cdot \hat Q_{r,a,b}
\end{align}
To finish the proof, notice that the powers of $X$ and $1-X$ are both non-negative for $r+1 \leq n \leq 2n$, and moreover
\begin{align}
\mathrm{deg} \left( X^{n-r-1} (1-X)^{2r-n}\cdot \hat Q_{r,a,b} \right)
\leq (n-r-1)+(2r-n) + (r_0-r) = r_0-1
\end{align}
Note also that the polynomial $\hat Q_{r,a,b}$ has integer coefficients, so the same must hold for $Q_{r_0,a_0,b_0}$. This completes the proof.
\end{proof}

We prove now a similar result for $M(X,Y,A,B)$.
To this end, we need the following two results from classical multivariate analysis: the \emph{generalized Leibniz product rule}:
$$\frac{\partial^n}{\partial x_1 \partial x_2 \cdots \partial x_n}(uv) = \sum_{\emptyset \subseteq S \subseteq [n]} \frac{\partial^{|S|}u}{\prod_{i \in S} \partial x_i} \cdot \frac{\partial^{|S^c|}v}{\prod_{i \in S^c} \partial x_i} ,$$
and the \emph{generalized Fa\`a di Bruno chain rule}:
$$\frac{\partial^n}{\partial x_1 \partial x_2 \cdots \partial x_n}(u \circ v) = \sum_{\pi \in \Pi_n} u^{(\# \pi)}(v)  \cdot \prod_{\lambda \in \pi}\frac{\partial^{|\lambda|}v}{\prod_{i \in \lambda} \partial x_i},$$
where, in the first formula, the sum runs over all subsets $S$ of $[n] = \{1, 2, \ldots, n\}$ and, in the second formula, the sum runs over all (possibly crossing) partitions $\pi$ of $[n]$.  As before, we denote by $\# \pi$ the number of blocks of the partition $\pi$, and we write $\lambda \in \pi$ for a block $\lambda$ of $\pi$. 

In the following statement we show that, after a change of variables,  the generating series $M$ has a simple form. The change of variables is motivated by the $Y=A=B=0$ series: $W \mapsto W/(1+W)^2$ is the functional inverse of the function $g_{0,0,0}$ from \eqref{eq:g000}.

\begin{proposition}\label{prop:coefficients-M}
The coefficients of the formal power series $M(X,Y,A,B)$, after the change of variable 
$\tilde M(W,Y,A,B) = M(W/(1+W)^2,Y,A,B)$, are as follows:
by writing $\tilde g_{r,a,b} (W) = g_{r,a,b} (W/(1+W)^2)$,
\begin{itemize}
\item For $r=0$,  we have $\tilde g_{0,0,0} (W) = W$ and $\tilde g_{0,a,b} (W) = 0$ for $(a,b) \not = (0,0)$.
\item For $r \geq 1$ and any $a,b$ such that $(r,a,b)$ is compatible, 
\begin{equation}\label{eq:coeff-tilde-M}
\tilde g_{r,a,b} (W) = \frac{W^{r+1}(1+W)P_{r,a,b}(W)}{(1-W)^{2r-1}},
\end{equation}
for some polynomial $P_{r,a,b}$ of degree at most $3r-3$. 
\end{itemize}
\end{proposition}
\begin{proof}
Notice that the set $M_{n,0,a,b}$ is empty unless $(a,b)=(0,0)$, so that
$|M_{n,0,0,0}| =\mathrm{Cat}_n$ for all $n \geq 1$.
This means that $1+ g_{0,0,0}(X)$ is the generating function of Catalan numbers so that
\begin{equation}\label{eq:g000}
g_{0,0,0}(X) = \frac{1-\sqrt{1-4X}}{2X}-1 
\end{equation}
By replacing $X$ by $W/(1+W)^2$ we can prove the first statement. 

We shall prove the general case by recurrence, as we did in Proposition \ref{prop:coefficients-K}. 
To do so, fix $(r_0,a_0,b_0)$ with $r_0 \geq 1$, 
and let us assume that the expression \eqref{eq:coeff-tilde-M} holds for all (compatible) triples $(r,a,b)$ with $(r,a,b) < (r_0,a_0,b_0)$.
Then, we can write, separating the case $(r,a,b)=(0,0,0)$,
$$K(X,Y,A,B) = f_{0,0,0}(X) + \sum_{(r,a,b) \neq (0,0,0)} f_{r,a,b}(X) \cdot Y^rA^aB^b,$$ 
so that, via Corollary \ref{cor:I-K-M}, we get
\begin{equation}\label{eq:tilde-M-K-main}
\tilde M =\underbrace{ f_{0,0,0} \left( \frac{W}{(1+W)^2}(1+\tilde M)\right)}_{\clubsuit} 
+  \sum_{(r,a,b) \neq (0,0,0)} \underbrace{ f_{r,a,b}\left[\frac{W}{(1+W)^2}(1+\tilde M) \right] Y^rA^aB^b}_{\spadesuit(r,a,b)}.
\end{equation}
Here and below we write $\tilde M = \tilde M (W,Y,A,B)$.
In order to find the coefficient $\tilde g_{r_0,a_0,b_0} (W) =[Y^{r_0}A^{a_0}B^{b_0}]\tilde M$, 
we shall take the derivative  of  \eqref{eq:tilde-M-K-main}, and \emph{then} set $Y=A=B=0$:
\eq{\label{eq:PD-M}
D_0[r_0,a_0,b_0] \tilde M = (r_0!)(a_0!)(b_0!) \cdot \tilde g_{r_0,a_0,b_0} (W)
}
First, let us deal with the first term of the RHS of \eqref{eq:tilde-M-K-main}: $\clubsuit$. 
Since $f_{0,0,0}(X) = X/(1-X)$ via  Proposition \ref{prop:coefficients-K}, the $p$-th derivative is given by 
\begin{align}\label{eq:D-f}
\forall p \geq 1, \qquad f_{0,0,0}^{(p)}(X) =\frac{ p! }{(1-X)^{p+1}}.
\end{align}
Using the Fa\`a di Bruno formula, the derivative of $\clubsuit$ reads, 
\begin{align} 
D[r_0,a_0,b_0] \, \clubsuit &=\sum_{\pi \in \Pi_{r_0+a_0+b_0}} f_{000}^{(\#\pi)}\left[\frac{W}{(1+W)^2}(1+\tilde M) \right]
\cdot \prod_{\lambda \in \pi}
D[\lambda_1,\lambda_2,\lambda_3]  \frac{W}{(1+W)^2}(1+\tilde M)\\
&= \sum_{\pi \in \Pi_{r_0+a_0+b_0}} f_{000}^{(\#\pi)}\left[\frac{W}{(1+W)^2}(1+\tilde M) \right]
\cdot \left( \frac{W}{(1+W)^2} \right)^{\# \pi}
\cdot \prod_{\lambda \in \pi}
D[\lambda_1,\lambda_2,\lambda_3]  \tilde M
\end{align}
Here, for a block $\lambda \in \pi$, 
we let $\lambda_1$ (resp.~$\lambda_2$ and $\lambda_3$) be the number of indices $i \in \lambda$  such that 
$x_i = Y$ (resp.~$x_i = A$ and $x_i = B$)  in the sense that we regard $\pi$ as a partition of
\eq{
(x_i) = (\underbrace{Y,\ldots,Y}_{r_0},\underbrace{A,\ldots,A}_{a_0},\underbrace{B,\ldots,B}_{b_0})
}
To calculate further we use the induction hypothesis. For $(\lambda_1,\lambda_2,\lambda_3) < (r_0,a_0,b_0)$ we have
\begin{equation}\label{eq:partial-derivative-tilde-M}
D_0[\lambda_1,\lambda_2,\lambda_3] \tilde M
= \lambda_1!\lambda_2!\lambda_3! \cdot \frac{W^{\lambda_1+1}(1+W)P_{\lambda_1, \lambda_2, \lambda_3}(W)}{(1-W)^{2\lambda_1-1}}
\end{equation}
Note that this vanishes based on the first statement of the current proposition (the compatibility condition),
but this fact will not be used. 
Also, the polynomial $P_{\lambda_1, \lambda_2, \lambda_3}$ is of degree at most $3\lambda_1 - 3$. 
Therefore, for $\pi \not = {1}_{r_0+a_0+b_0}$ by the induction hypothesis we have
\begin{align}\label{eq:FB-part}
\prod_{\lambda \in \pi} 
D_0[\lambda_1,\lambda_2,\lambda_3] \tilde M
=  \frac{W^{r_0 + \# \pi}(1+W)^{\# \pi}}{(1-W)^{2r_0 - \# \pi}}\cdot  \hat P_\pi (W) 
\end{align}
If $\#\pi >r_0$ this quantity vanishes  because one of the blocks of $\pi$ must give $0$ in \eqref{eq:partial-derivative-tilde-M},
but this fact does not make any difference in the current paper. More importantly, 
\begin{align}
\mathrm{deg} (\hat P_\pi) \leq 3 r_0 - 3\# \pi 
\end{align}
On the other hand, for $\pi  = {1}_{r_0+a_0+b_0}$ we have $(r_0!)(a_0!)(b_0!) \cdot \tilde g_{r_0,a_0,b_0} (W)$ instead.

Now we are ready to analyze $D_0 [r_0,a_0,b_0] \, \clubsuit$.
Since $\tilde M (W,0,0,0) = W$ from the first claim, by using \eqref{eq:D-f} we have
\begin{align}
f_{000}^{(\#\pi)}\left[\frac{W}{(1+W)^2}(1+\tilde M) \right] = (\# \pi)! \cdot (1+W)^{\#\pi +1}
\end{align} 
Here, we used the following equality:
\eqb{\label{eq:variable-f}
\frac{W}{(1+W)^2} (1+\tilde M(W,0,0,0)) = \frac{W}{1+W} 
}
Then, 
\eq{\label{eq:club-estimate}
D_0 [r_0,a_0,b_0] \, \clubsuit 
&=(r_0!)(a_0!)(b_0!)W \cdot \tilde g_{r_0,a_0,b_0} (W) +
\sum_{\substack{\pi \in \Pi_{r_0+a_0+b_0}\\ \pi \not = {1}_{r_0+a_0+b_0}}} 
(\# \pi)!  \cdot \frac{ W^{r_0 + 2 \# \pi} (1+W) }{(1-W)^{2r_0- \#\pi}} \cdot  \hat P_\pi (W)\notag\\
&= (r_0!)(a_0!)(b_0!)W \cdot \tilde g_{r_0,a_0,b_0} (W) +
\frac{W^{r_0+1}(1+W)}{(1-W)^{2_0-2}} \cdot \hat{\hat P} (W)
}
where  $\# \pi  \geq 2$ and
\begin{align}
\hat{\hat P}(W) = \sum_{\substack{\pi \in \Pi_{r_0+a_0+b_0}\\ \pi \not = {1}_{r_0+a_0+b_0}}} 
(\# \pi)! \cdot W^{2 \# \pi-1}(1-W)^{ \#\pi - 2} \cdot  \hat P_\pi (W)
\end{align}
This is a polynomial because the powers in the above formula are all non-negative, and moreover,
\begin{align}
\mathrm{deg} (\hat{\hat P}) \leq (2 \# \pi-1) + (\#\pi - 2) + (3 r_0 - 3\# \pi ) = 3r_0 - 3.
\end{align}

\medskip

Let us now focus on the second term in \eqref{eq:tilde-M-K-main}: $\spadesuit(r,a,b)$.
For fixed $(r,a,b)$ and $(r_0,a_0,b_0)$
\begin{align*}
D_0[r_0,a_0,b_0]\, \spadesuit(r,a,b)& = \sum_{r'=0}^{r_0}\sum_{a'=0}^{a_0}\sum_{b'=0}^{b_0} 
c_{r'a'b'}  \cdot \left(D_0[r',a',b'] \,  f_{rab}\left[\frac{W}{(1+W)^2}(1+\tilde M) \right]  \right) \\
&\hspace{40mm}\cdot \left( D_0[r_0-r',a_0-a',b_0-b'] (Y^r,A^a,B^b) \right) \\
& =c_{r_0-r,a_0-r,b_0-b}  \cdot D_0[r_0-r,a_0-a,b_0-b] \,  f_{r,a,b}\left[\frac{W}{(1+W)^2}(1+\tilde M) \right] 
\end{align*}
where $c_{r',a',b'}$ is some combinatorial factor.
Here, we used the generalized Leibniz product rule. 

To continue our calculation, let us first compute the derivative (in $X$) of the function $f_{r,a,b}$ for $r \geq 1$ in \eqref{eq:coeff-K}. 
For an arbitrary order of derivation $p \geq 1$, we have
$$f_{r,a,b}^{(p)}(X) = \sum_{p = s+t+u} X^{r+1-s}(1-X)^{-2r-1-t}Q_{r,a,b,s,t,u}(X),$$
where $Q_{r,a,b,s,t,u}$ is a polynomial of degree at most $r-1-u$, 
which also incorporates the combinatorial factors obtained from the derivation of the powers of $X$ and $1-X$. 
Note that $f_{rab}^{(p)}$ vanishes unless $s \leq r+1$ and $u \leq r-1$.

Then, by using generalized Fa\`a di Bruno chain rule together with \eqref{eq:variable-f} and \eqref{eq:FB-part} (replacing $r_0$ by $r_0-r$) 
we have, with $l= r_0+a_0+b_0 - (r+a+b)$,
\eqb{
D_0&[r_0-r,a_0-a,b_0-b] \, f_{r,a,b}\left[\frac{W}{(1+W)^2}(1+\tilde M) \right] \\
&= \sum_{\pi \in \Pi_l}  f_{r,a,b}^{(\#\pi)}\left[\frac{W}{(1+W)^2}(1+\tilde M) \right] \cdot  \prod_{\lambda \in \pi} D[\lambda_1,\lambda_2,\lambda_3]  \frac{W}{(1+W)^2}(1+\tilde M) \\
&= \sum_{\pi \in \Pi_l}  \sum_{\# \pi = s+t+u} \left( \frac{W}{1+W} \right)^{r+1-s} \left( \frac{1}{1+W} \right)^{-2r-1-t} \\
&\hspace{2cm} \cdot Q_{r,a,b,s,t,u}\left( \frac{W}{1+W} \right)
\cdot \left( \frac {W}{(1+W)^2} \right)^{\#\pi}
\cdot \frac{W^{r_0-r + \# \pi}(1+W)^{\# \pi}}{(1-W)^{2r_0-2r - \# \pi}} \hat P_\pi (W) \\
&=\frac{W^{r_0+1}(1+W)}{(1-W)^{2r_0 - 2}} \cdot \tilde P_{r,a,b} (W)
}
where
\eq{
\tilde P_{r,a,b} (W)=  \sum_{\pi \in \Pi_l}  \sum_{\# \pi = s+t+u} W^{2\# \pi -s}{(1-W)^{\# \pi+ 2r  -2}}\cdot R_{r,a,b,s,t,u}(W) \cdot \hat P_\pi (W)
}
with 
\eq{
R_{r,a,b,s,t,u}(W) = (1+W)^{r-1-u}\cdot Q_{r,a,b,s,t,u}\left( \frac{W}{1+W} \right).
} 
Note that $R_{r,a,b,s,t,u}$ is a polynomial of degree at most $r-1-u$.
Hence $\tilde P_{r,a,b}$ is also a polynomial such that 
\eqb{
\mathrm{deg} (\tilde P_{r,a,b})& \leq (2\# \pi -s)  + (\# \pi +2r -2) +(3r_0-3r-3\#\pi)+(r-1-u)  \\
&=3r_0-3-s-u \leq 3r_0 -3.
}
Therefore
\eq{\label{eq:spade-estimate}
 \sum_{(r,a,b) \neq (0,0,0)}  D_0[r_0,a_0,b_0]\, \spadesuit(r,a,b) 
 =\frac{W^{r_0+1}(1+W)}{(1-W)^{2r_0 - 2}} \cdot \tilde{\tilde P} (W)
}
with the degree of $ \tilde{\tilde P}$ at most $3r_0-3$.

Putting \eqref{eq:PD-M}, \eqref{eq:club-estimate} and \eqref{eq:spade-estimate} together, 
\eqref{eq:tilde-M-K-main} gives the following:
\eq{
(r_0!)(a_0!)(b_0!) \cdot (1-W)\cdot \tilde g_{r_0,a_0,b_0} (W) 
= \frac{W^{r_0+1}(1+W)}{(1-W)^{2r_0 - 2}} \cdot\left[\hat{\hat P}  (W)+\tilde{\tilde P} (W) \right]
}
This completes the proof. 
\end{proof}

\begin{remark}\label{rem:integer-coeff-M}
We believe that the polynomials $P_{r,a,b}$ have integer coefficients, and in fact this is the case for $r \leq 6$ (see Section \ref{sec:numerics}). 
Perhaps one could try and show this by a more careful analysis of the combinatorial factors appearing in the proof.
\end{remark}

From the previous proposition, our main theorem states as:
\begin{theorem}\label{thm:meanders-generating-function}
For any fixed $r \geq 1$ there exists a polynomial $\tilde P_r$ of degree at most $3r-3$ such that the generating function of the number of meanders on $2n$ points with $n-r$ loops 
$$F_r(t) = \sum_{n=r+1}^\infty M_n^{(n-r)}t^n,$$
with the change of variables $t=w/(1+w)^2$, reads
\begin{equation}\label{eq:F-r-w}
F_r(t) = \sum_{n=r}^\infty M_n^{(n-r)} \frac{w^n}{(1+w)^{2n}} = \frac{w^{r+1}(1+w)}{(1-w)^{2r-1}} \tilde P_r(w).
\end{equation}
\end{theorem}
\begin{proof}
The claim follows directly from Proposition \ref{prop:coefficients-M}, by setting $A=B=1$ and writing
$$\tilde P_r = \sum_{a,b \, : \, (r,a,b) \text{ compatible}} P_{r,a,b}.$$
\end{proof}
\begin{remark}\label{rem:rel-conj}
One can obtain $M_n^{(n-r)}$, the number of meandric systems on $2n$ points with $n-r$ loops, from the series $F_r$, 
after a change of variables and a series expansion. It would be interesting to relate the explicit form of $F_r$ from \eqref{eq:F-r-w} to the conjecture from \cite[Equation (2.4)]{difrancesco1997meander}.
\end{remark}

As a corollary of the formula \eqref{eq:F-r-w} for the generating series, we obtain the asymptotic behavior of the meandric numbers $M_n^{(n-r)}$. 

\begin{corollary}\label{cor:asympt}
For any fixed $r \geq 1$, assuming that $\tilde P_r(1) \neq 0$,  the number of meandric systems on $2n$ points having $n-r$ loops has the following asymptotic behavior as $n \to \infty$:
\begin{equation}
 M_n^{(n-r)} \sim \frac{\tilde P_r(1)}{2^{2r-2}\Gamma((2r-1)/2)}4^n n^{(2r-3)/2}.
 \end{equation}
 \end{corollary}
 \begin{proof}
 The generating function $F_r$ from \eqref{eq:F-r-w} is analytic on $\mathbb C \setminus [1/4, \infty)$, hence the exponential growth of the meandric numbers is $4^n$, see \cite[Theorem IV.7]{flajolet2009analytic}. For the more precise statement, we use the transfer results from \cite[Section VI]{flajolet2009analytic}. Note that the behavior of $w(x)$ at $x \to 1/4$ and of $F_r(w)$ at $w \to 1$ are given respectively by
 \begin{align*}
 w(x) &\sim 1- 4 \sqrt{1/4-x}\\
 F_r(w) &\sim 2\tilde P_r(1) (1-w)^{1-2r}.
 \end{align*}
 Hence, in a ``Camambert region'' with the opening at $x=1/4$, we have the following equivalent when $x \to 1/4$
 $$F_r(x) \sim \frac{2 \tilde P_r(1)}{4^{2r-1}}(1/4-x)^{-(2r-1)/2}.$$
 By \cite[Theorem VI.1]{flajolet2009analytic}, it follows that
 $$M_n^{(n-r)} \sim  \frac{2 \tilde P_r(1)}{4^{2r-1}} \frac{4^{(2r-1)/2}}{\Gamma((2r-1)/2)} 4^n n^{(2r-3)/2},$$
 which is the announced result.
 \end{proof}
 
 \bigskip

We end this section with some comments relating our approach to the meander generating series with previous results. 
Lando and Zvonkin were the first ones to study irreducible meandric systems \cite{lando1993plane}. 
They show that the following formal equality holds:
\begin{equation}\label{eq:LZ}
B(x) = N(xB^2(x)),
\end{equation}
where $B(x)$ is the generating series for the square Catalan numbers $B(x) = \sum_{n=0}^\infty \mathrm{Cat}_n^2 x^n$ and 
$N$ is the generating series for irreducible meandric systems. In our notation, $N(X) = 1+I(X)$ and $B(X) = 1+M(X)$.
Then, our formula gives theirs by setting $Y=A=B=1$ where we just split the reduction in two steps. 
Indeed, the series $I,K,M$ are related by the relations
\begin{align}
\label{eq:I-to-K} K(X) &= I(X(1+K(X)))\\
\label{eq:K-to-M} M(X) &= K(X(1+M(X))).
\end{align}
Plugging \eqref{eq:I-to-K} into \eqref{eq:K-to-M}, we get
\begin{align*}
M(X) &=  K(X(1+M(X)))\\
&= I \left[ X(1+M(X)) \cdot \left( 1 + K(X(1+M(X))) \right) \right]\\
&= I \left[ X(1+M(X))^2\right],
\end{align*}
which is precisely \eqref{eq:LZ}. Similarly, the computations in \cite[Remark 4.5]{nica2016free} can be shown to be equivalent to the special case $Y=A=B=1$ of Theorem \ref{thm:M-I} in a similar straightforward fashion. 
Although the derivations in \cite{lando1993plane,nica2016free} seem simpler, since they only require one implicit functional equation to be solved, 
the quadratic term appearing in \eqref{eq:LZ} makes this equation more complicated to deal with.

\section{Exact formulas for small values of \texorpdfstring{$r$}{r}}\label{sec:numerics}

We gather in this section the formulas for the generating functions of the numbers of meandric systems on $n$ points with $n-r$ loops, for small values of $r$ ($r \leq 6$). Let us emphasize that our method could be used, in principle, to obtain the generating functions for all (fixed, but arbitrarily large) values of $r$; we are limited by the following computational tasks:
\begin{enumerate}
\item computing the number of irreducible meandric systems of type $(r,a,b)$
\item performing the formal power series inversion in \eqref{eq:I-to-K}, \eqref{eq:K-to-M}.
\end{enumerate}
We have implemented the above computational steps and automated the computation of the generating functions. First, a \textsc{C} program computes all the irreducible  meandric systems of size $p$, storing the results for later use. Then, a symbolic \textsc{Mathematica} function computes automatically the generating function (for given $r$), using the irreducible meandric system data. All the software is available at \cite{num}. We think that our crude computer implementation could be optimized to reach larger values of $r$. 

Finally, let us once more make the observation that our method is not well suited to tackle the (most important) problem of enumerating connected meanders. Connected meanders correspond to taking $r=n-1$, while in our method $r$ is a fixed parameter which is not allowed to grow with $n$. 

\begin{proposition}
The polynomials $\tilde P_r$ appearing in the generating function \eqref{eq:F-r-w} for meandric systems on $2n$ points having $n-r$ loops are as follows
\begin{align*}
\tilde P_1(w) &= 2\\
\tilde P_2(w) &= 4 w^3-12 w^2+4 w+8\\
\tilde P_3(w) &= 18 w^6-92 w^5+134 w^4+8 w^3-146 w^2+52 w+42\\
\tilde P_4(w) &= 112 w^9-770 w^8+1864 w^7-1344 w^6-1656 w^5+3052 w^4-520 w^3-1440 w^2+520 w+262\\
\tilde P_5(w) &= 820 w^{12}-7052 w^{11}+23264 w^{10}-31788 w^9-3108 w^8+60568 w^7-54912 w^6-16808 w^5\\&\quad+48012 w^4-11660 w^3-13664 w^2+4948 w+1828\\
\tilde P_6(w) &= 6632 w^{15}-68322 w^{14}+283820 w^{13}-558256 w^{12}+311016 w^{11}+798210 w^{10}-1587476 w^9\\&\quad+556540 w^8+1213592 w^7-1278814 w^6-76668 w^5+652408 w^4-181480 w^3-129026 w^2\\&\quad+46692 w+13820.
\end{align*}
\end{proposition}

Note that all the polynomials above have even integer coefficients. Although we have not proved this fact (see Remark \ref{rem:integer-coeff-M}), the factor $2$ appearing in front of each coefficient of $\tilde P$ has a simple interpretation: for every $r \geq 1$, for each pair $(\alpha, \beta)$ contributing to the series $M$, there is the pair $(\beta, \alpha) \neq (\alpha, \beta)$ which also contributes. 

Plugging these values above into Corollary \ref{cor:asympt}, we obtain the exact asymptotic behavior of meandric numbers for $r \leq 6$. 

\begin{corollary}\label{cor:asympt-precise-6}
The first $6$ series of meandric numbers $M_n^{(n-r)}$ have the following asymptotic behavior as $n \to \infty$
\begin{align*}
M_n^{(n-1)} &\sim \frac{2}{\sqrt \pi} 4^n n^{-1/2}\\
M_n^{(n-2)} &\sim \frac{2}{\sqrt \pi} 4^n n^{1/2}\\
M_n^{(n-3)} &\sim \frac{4}{3\sqrt \pi} 4^n n^{3/2}\\
M_n^{(n-4)} &\sim \frac{2}{3\sqrt \pi} 4^n n^{5/2}\\
M_n^{(n-5)} &\sim \frac{4}{15\sqrt \pi} 4^n n^{7/2}\\
M_n^{(n-6)} &\sim \frac{4}{45\sqrt \pi} 4^n n^{9/2}.
\end{align*}
\end{corollary}

\bibliography{ref}{}
\bibliographystyle{alpha}

\end{document}